\documentclass[12pt]{article}

\usepackage[russian,ukrainian,english]{babel}

\usepackage{amssymb}
\usepackage{graphics}

\newcommand{\qed}{$\;\;\;\Box$}
\newenvironment{proof}{\par\smallbreak{\sl\bf Proof.~}}
{\unskip\nobreak\hfill \qed \par\medbreak}

\newcounter{claim}
\renewcommand{\theclaim}{\arabic{claim}}
{\par\medskip\par}

{\qed\par\smallbreak}
\newcommand{\N}{{\mathbb N}}
\newcommand{\R}{{\mathbb R}}

\newcommand{\LL}{{\cal L}}

\newcommand{\beq}{\begin{equation}}
\newcommand{\ee}{\end{equation}}

\renewcommand{\d}{\partial}

\newtheorem{thm}{Theorem}[section]
\newtheorem{lem}[thm]{Lemma}
\newtheorem{defn}[thm]{Definition}
\newtheorem{cor}[thm]{Corollary}
\newtheorem{rem}[thm]{Remark}
\newtheorem{ex}[thm]{Example}

\newcommand{\al}{\alpha}
\newcommand{\be}{\beta}
\newcommand{\ga}{\gamma}

\newcommand{\eps}{\varepsilon}
\newcommand{\vphi}{\varphi}
\newcommand{\la}{\lambda}
\newcommand{\om}{\omega}
\newcommand{\io}{\iota}

\newcommand{\reff}[1]{(\ref{#1})}      

\newcommand{\diag}{\mathop{\rm diag}\nolimits}

\date{}

\title{
Perturbations of  superstable linear hyperbolic systems}

\newcounter{thesame}
\author{
I.~Kmit
\thanks{Institute of Mathematics, Humboldt University of Berlin. On leave from the
Institute for Applied Problems of Mechanics and Mathematics,
Ukrainian National Academy of Sciences. {\small   E-mail:
{\tt kmit@mathematik.hu-berlin.de}}}
\ \ \ N.~Lyul'ko \thanks{Sobolev Institute of Mathematics, Russian Academy of Sciences and
Novosibirsk State University, Russia.
{\small   E-mail:
{\tt natlyl@mail.ru}}
}}

\begin{document}

\maketitle

\begin{abstract}
\noindent 
The paper deals with  initial-boundary value problems
 for linear non-autonomous first order hyperbolic systems
whose solutions  stabilize to zero in a finite time.
We prove that problems in this class remain 
 exponentially stable in $L^2$ as well as in $C^1$ under  small bounded perturbations. 
To show this for $C^1$, we prove a general smoothing result implying that
the solutions to the perturbed problems
 become eventually $C^1$-smooth for any $L^2$-initial data.
\end{abstract}

\emph{Key words:} 
 first order hyperbolic systems, smoothing boundary conditions,
superstability,
exponential  stability,  bounded perturbations, evolution family

\emph{Mathematics Subject Classification: 35B20, 35B35, 35B40, 35B45,   35B65,  35L50}

\section{Introduction}\label{sec:intr} 

A linear system
\begin{equation}\label{ss0}
\frac{d}{dt}x(t)=A(t)x(t), \quad x(t)\in X \quad (0\le t \le \infty),
\end{equation}
on a Banach space $X$
is called  \textit{exponentially stable} if there exist positive reals
$\ga$ and $M=M(\gamma)$ such that every solution $x(t)$ 
satisfies the estimate
\beq\label{cor1}
\|x(t)\|\le M e^{-\gamma t} \|x(0)\|, \quad t\ge 0,
\ee
where  $\|\cdot\|$ denotes the norm in $X$. 

The papers  \cite{bal99,bal105,cr13} address a stronger property of exponentially stable systems, known as superstability. They consider the Cauchy problem for the  {\it autonomous} version of \reff{ss0},
where $A(t)=A$ does not depend on $t$. Moreover,
$A: X \rightarrow X$ is supposed to be  the infinitesimal generator of a strongly continuous semigroup  $T(t)$; see \cite{hp57,Paz}. 
A semigroup $T(t)$ is called {\it superstable}  \cite{bal99,bal105,lum01,rw95} 
if its {\it stability index}
is $-\infty$, that is
$$
\lim_{t\to \infty}\frac{\log\|T(t)\|}{t}=-\infty.
$$
In this case
the system (\ref{ss0}) is called  superstable also.
The superstability property 
implies that the system is exponentially stable and, moreover, the estimate \reff{cor1}
holds for every $\gamma> 0$. 
For a  superstable system the resolvent $R(\lambda;A) $ of the operator $A$, which is
defined by the formula
$$
R(\lambda; A)x=\int_0^\infty\,e^{-\lambda t} T(t)x\,dt,\quad x\in X,
$$
is an entire function of the complex parameter $\lambda$, and the spectrum of the operator 
$A$, which we denote by $\sigma(A)$, is empty. 
The superstability  property makes sense only for
systems in infinite-dimensional Banach spaces, since  any linear operator $A: X \rightarrow X $ in a finite-dimensional space $X$ 
has a non-empty point spectrum, and the stability index is equal to
 the maximum of the real parts 
of the eigenvalues of $A$.
 
An important subclass of  superstable systems, that will be studied in the present paper,
 consists of the systems
whose solutions  stabilize to zero 
after some  time. The time of the stabilization is called  a {\it finite time extinction}.
The simplest example is given by the initial-boundary value problem \cite{hp57}
$$
\begin{array}{ll}
u_t+u_x=0,& (x,t)\in (0,1)\times(0,\infty),\\
u(x,0)=u_0(x),& x\in [0,1],\\
u(0,t)=0,& t\in(0,\infty).
\end{array}
$$
It is easy to check that all solutions to this problem stabilize to zero for $t>1$. 
Similar  examples for the wave equation  are given in \cite{Cox,Kom,Majda}.

Here we address superstable initial-boundary value problems with finite time extinction for linear 
{\it non-autonomous} hyperbolic systems. 
We consider bounded perturbations of such problems and investigate
 the asymptotic behavior of their solutions. In contrast to the autonomous case, which is 
well-studied, the non-autonomous case has been considered in the literature only episodically.

The recent papers \cite{lakra,pavel,
perroll,slx13,udw05,udw12} are devoted to superstable hyperbolic models 
intensively used
 in the control theory. By  introducing control parameters in the boundary conditions 
and/or 
in the  coefficients of the differential equations, such
  systems can be stabilized to a desired state in a finite time, 
which,  from the physical  point of view, 
is even more preferable than the infinite time stabilization.
Superstable hyperbolic systems are usually supplemented with the so-called 
quiet boundaries \cite{udw05}, where
the influence of the reflected waves is minimized or even neglected. Mathematically, 
 the quiet boundaries are described by means of the so-called {\it smoothing boundary
conditions}, that also will be considered in the present paper.

The paper is organized as follows. 
 In Section \ref{sec:statem}  we state the problem and formulate our results
about the existence of evolution families, smoothing properties of solutions, 
and exponential stability
 of solutions.
Some comments and examples related to  applications  
are  given in Section \ref{sec:exam}.
In particular, Example \ref{exx1} shows  how our results obtained for linear systems can be applied to
show the exponential stability of solutions to nonlinear problems.
 Using a priori estimates, in Section \ref{sec:abstr} we prove that the problem
under consideration generates an
evolution family on  $L^2(0,1)^n$.
In Section \ref{sec:smoothing} 
we extend the results of  \cite{Elt,kmit,Km,LavLyu}  by 
 showing that  boundary operators of reflection  type
  cause the smoothing effect in the sense that  the solutions reach the $C^1$-regularity
 for any $L^2$-initial data after some time. Furthermore, we provide general conditions  for that.
We also discuss  the relationship
between the smoothing property and the stabilization to zero.
Finally, in Section \ref{sec:perturb},
using the  variation of constants formula,
 we prove that superstable hyperbolic operators remain exponentially stable
 under small  bounded 
perturbations.

\paragraph{Previous work}
Our stability analysis  shows that the exponential decay rate $\ga$ of solutions to the perturbed 
problems can be arbitrary large provided
the perturbations are sufficiently small. Similar questions are addressed for the  wave equation perturbed by  
velocity damping, see e.g.  \cite{gugat,gugat1} and references therein.

In \cite{hp57} the authors  investigate a relationship between the asymptotic behavior 
of the spectrum and the resolvent of an operator $A$ and the stability 
of the corresponding semigroup. 
In \cite{kr},  properties of the resolvent $R(\lambda;A)$ are used to  establish
a criterion that a  strongly continuous semigroup $T(t)$ is eventually vanishing or, the same, is
nilpotent \cite{cr13}. Stability of nilpotent
semigroups is investigated in \cite{chen}.
Properties of the strongly continuous semigroup $T(t)$ generated by $A$ are used in \cite{cr13} for
 giving criteria that the autonomous system \reff{ss0} is asymptotically stable, 
superstable, or stabilizes 
to zero in a finite time.

In  \cite{Elt} the second author considers  initial-boundary value problems of the type \reff{ss0} for 
 autonomous strictly hyperbolic systems. Analyzing the resolvent of $A$,
she describes a class of boundary conditions for which solutions 
eventually stabilize to zero. The condition  $\sigma(A)=\emptyset$ gives
 a criterion for stabilization for a decoupled hyperbolic system. It is also shown that the 
stabilization property is closely related to  increasing  smoothness of solutions to the 
perturbed autonomous problems. 
 In the  autonomous  strictly hyperbolic case  the smoothing effect is addressed
 in \cite{Elt,LavLyu,Lyu}, while in the  non-autonomous weakly hyperbolic case it is investigated  in \cite{kmit,Km}.

A comprehensive review of the available results on asymptotic
 behavior of solutions to linear and quasi-linear hyperbolic problems can be 
found in \cite{bastin,Coron,gugat1}.

\section{Problem setting and the main results}\label{sec:statem} 
\renewcommand{\theequation}{{\thesection}.\arabic{equation}}
\setcounter{equation}{0}

\paragraph{Notation} Set 
 $\Pi=\{(x,t):\,0< x< 1,\, t\in\R\}$.
Let $BC(\Pi)$ (respectively, $BC(\R)$) denote the Banach space of all bounded and continuous 
maps  $u: \overline\Pi\to \R$   (respectively, $u: \R\to \R$), with the norm 
$$
\|u\|_\infty=\sup_{(x,t)\in\Pi}|u| \quad  (\mbox{respectively, }\|u\|_\infty=\sup_{t\in\R}|u|).
$$
Furthermore, let   $BC^1(\Pi)$ denote the Banach space of all  $u\in BC(\Pi)$ 
such that $\d_xu\in BC(\Pi)$ 
and  $\d_tu\in BC(\Pi)$,    with the norm 
$$
\|u\|_1=\|u\|_\infty+\|\d_xu\|_\infty+\|\d_tu\|_\infty.
$$  
 
If $X$ is a Banach space, 
then the $n$-th Cartesian power $X^n$ is considered  a Banach space with the  norm
$$
\|u\|_{X^n}= \max_{i\le n} \|u_i\|_{X}.
$$ 

As usual, by $\LL(X,Y)$ we denote the space of linear
bounded operators from a Banach space $X$ into a Banach space $Y$, and write $\LL(X)$ for $\LL(X,X)$.

\paragraph{Problem setting}
In the strip $\Pi$ we consider the following decoupled non-autonomous 
hyperbolic system:
\begin{equation}\label{ss2}
\partial_tu  + a(x,t)\partial_xu + b_d(x,t)u = 0, \quad (x,t)\in \Pi,
\end{equation}
and its perturbed version
\begin{equation}\label{ss7}
\partial_tu  + a(x,t)\partial_xu + (b_d(x,t)+\tilde{b}(x,t))u = 0, \quad (x,t)\in \Pi,
\end{equation}
where $n\ge 2$, $u=(u_1,\ldots,u_n)$ is  a vector of real-valued functions,
$a=\diag(a_1,\dots,a_n)$ and $ b_d=\diag(b_1,\dots,b_n)$ are diagonal matrices,
and $\tilde{b}$ is an  $(n\times n)$-matrix with entries $\tilde b_{jk}$.
The systems \reff{ss2} and \reff{ss7} will be endowed with 
the reflection boundary conditions 
\beq\label{ss5}
\begin{array}{l}
\displaystyle
u_j(0,t) = 
\sum\limits_{k=1}^mp_{jk}u_k(1,t) + \sum\limits_{k=m+1}^np_{jk}u_k(0,t),
\quad  1\le j\le m,
\nonumber\\
\displaystyle
u_j(1,t) = \sum\limits_{k=1}^mp_{jk}u_k(1,t)+ \sum\limits_{k=m+1}^np_{jk}u_k(0,t),
\quad   m< j\le n,
\nonumber
\end{array}
\ee
where the integer $m$ is fixed in the range $0\le m \le n$ and the reflection coefficients $p_{ij}$ are real constants.

Suppose that 
\begin{equation}\label{ss44}
a_j, b_{j}, \tilde b_{jk}\in BC^1(\Pi)   \mbox{ for all }
j\le n
\end{equation}
 and
\begin{equation}\label{ss4}
\inf_{(x,t)\in \overline\Pi}a_j\ge \Lambda_0 \mbox{ for all } j\le m\quad 
\mbox{ and }\quad 
\sup_{(x,t)\in \overline\Pi}a_j\le -\Lambda_0 \mbox{ for all } j>m
\end{equation}
for some $\Lambda_0>0$. Condition (\ref{ss4}) ensures that
the hyperbolic systems (\ref{ss2}) and \reff{ss7}
are  non-degenerate and that all their characteristics are uniformly bounded
in  $\overline\Pi$. 

Given $\tau\in\R$, we also consider  the initial conditions at $t=\tau$, namely
\beq\label{ss11}
u(x,\tau) = \vphi(x), \quad x\in[0,1].
\ee
Set 
$$
\Pi_\tau=\{(x,t): 0< x< 1,\, t>\tau\}.
$$

   We are interested in the long time behavior of solutions
to the perturbed problem (\ref{ss7}), (\ref{ss11}), (\ref{ss5}) in $\Pi_\tau$,
under the assumption that solutions to the unperturbed problem (\ref{ss2}), (\ref{ss11}), (\ref{ss5}) in $\Pi_\tau$
stabilize to zero in a finite time. 
We will formulate our results in terms of evolution families  generated by the unperturbed 
and the perturbed problems; this concept will be introduced below.
In fact, we will prove the existence of an evolution family for the general linear first order 
hyperbolic system
 \begin{equation}\label{ss10}
\partial_tu  + a(x,t)\partial_xu + b(x,t)u = 0, \quad (x,t)\in \Pi,
\end{equation}
subjected to the boundary conditions  (\ref{ss5}). We assume that  $b$ is an  $(n\times n)$-matrix with entries
\begin{equation}\label{ss8}
b_{ij}\in BC^1(\Pi)\quad \mbox{ for all } i,\,j\le n.\end{equation}
The well-posedness
of the problem  (\ref{ss10}), (\ref{ss11}), (\ref{ss5}) in $\Pi_\tau$ in the spaces of continuous and 
continuously differentiable functions is investigated in   \cite{Aboliny,ijdsde}.

\paragraph{Characteristics and integral representation}

 For given $j\le n$,  $x \in [0,1]$, and $t \in \R$, the $j$-th characteristic of \reff{ss10}
passing through the point $(x,t)$ is defined
as the solution $\xi\in [0,1] \mapsto \om_j(\xi,x,t)\in \R$ to the initial value problem
\beq\label{char}
\partial_\xi\om_j(\xi,x,t)=\frac{1}{a_j(\xi,\om_j(\xi,x,t))},\;\;
\om_j(x,x,t)=t.
\ee
The assumption  (\ref{ss4}) implies that, if $(x,t)\in\overline\Pi_\tau$, then
 the characteristic curve $\theta=\om_j(\xi,x,t)$ 
reaches the
boundary of $\Pi_\tau$ in two points with distinct ordinates. Let $x_j(x,t)$
denote the abscissa of that point whose ordinate is smaller. 
The condition (\ref{ss4}) ensures that  the value of  $x_j(x,t)$ 
does not depend on $x,t$ if $t>\tau+\frac{1}{\Lambda_0}$. Therefore, 
$ x_j(x,t)$ in this range takes on the constant value
\beq\label{*k}
x_j=\left\{
 \begin{array}{rl}
 0 &\mbox{if}\ 1\le j\le m,\\
 1 &\mbox{if}\ m<j\le n.
\end{array}
\right.
\ee
Let
$$c_j(\xi,x,t)=\exp \int_x^\xi
\left(\frac{b_{jj}}{a_{j}}\right)(\eta,\om_j(\eta,x,t))\,d\eta,\quad
d_j(\xi,x,t)=\frac{c_j(\xi,x,t)}{a_j(\xi,\om_j(\xi,x,t))}.
$$
We define  a linear bounded operator
 $P\in\LL\left(BC(\Pi)^n, BC(\R)^n\right)$ by
\beq\label{ss12}
\displaystyle
\left(Pu\right)_j(t) = \sum\limits_{k=1}^mp_{jk}u_k(1,t)+\sum\limits_{k=m+1}^np_{jk}u_k(0,t),
\quad   j\le n.
\ee

Straightforward calculations  show that a $C^1$-map $u:\overline{\Pi}_\tau \to \R^n$
is a solution to 
the problem (\ref{ss10}), (\ref{ss11}), (\ref{ss5})  if and only if
it satisfies the following system of integral equations
\begin{eqnarray}
\label{rep1}
\lefteqn{
u_j(x,t)=
\left(Qu\right)_j(x,t)
}\nonumber\\
&&-\int_{x_j(x,t)}^x d_j(\xi,x,t)\sum_{k\not=j} b_{jk}(\xi,\om_j(\xi,x,t))u_k(\xi,\om_j(\xi,x,t))d\xi,
\quad j\le n,
\end{eqnarray}
where the affine operator $Q$   is defined by
\begin{eqnarray}
\label{Q}
(Qu)_j(x,t)=
\left\{\begin{array}{lcl}
\displaystyle c_j(x_j(x,t),x,t) \left(Pu\right)_j(\om_j(x_j(x,t),x,t))
& \mbox{if}&  x_j(x,t)=0 \mbox{ or } x_j(x,t)=1 \\
c_j(x_j(x,t),x,t)\vphi_j(x_j(x,t))      & \mbox{if}&  x_j(x,t)\in(0,1),
\end{array}\right.
\end{eqnarray}
on a subset of  $C(\overline\Pi_\tau)^n$ of functions  satisfying the initial condition
(\ref{ss11}).
A continuous function $u$ satisfying  \reff{rep1} in $\overline{\Pi}_\tau$ 
is called a {\it continuous solution} to (\ref{ss10}), (\ref{ss11}), (\ref{ss5}).

\paragraph{Our results}
In Section 3 we prove that the problem (\ref{ss10}), (\ref{ss11}), (\ref{ss5}), where
$\tau\in\R$ is  an arbitrary fixed initial time,
generates an evolution family
$\{U(t,\tau)\}_{t\ge\tau}$
 mapping an initial function $\vphi$ given at time $\tau$ into the solution
$U(t,\tau)\varphi$ of the problem  (\ref{ss10}), (\ref{ss11}), (\ref{ss5})
at time~$t$. We now define this concept formally.

\begin{defn}\label{def_evolfam} \cite{Paz}
A two-parameter family $\{U(t,\tau)\}_{t\ge\tau}$ 
of linear bounded operators on a Banach space $X$
is called an {\rm evolution family} 
if it satisfies the following properties:
\begin{itemize}
\item
$U(\tau,\tau) =I$ and $U(t,s)U(s,\tau) =  U(t,\tau)$ for all
$t \ge s \ge\tau$;
\item
the map $(t,\tau)\in\R^2\rightarrow U(t,\tau)\in\LL(X)$ is strongly continuous for all
$t\ge \tau$.
\end{itemize}
\end{defn}

\begin{defn}\label{def_eo}   \cite{Paz}
An  evolution  family  $\{U(t,\tau)\}_{t\ge\tau}$   on a Banach space $X$
 is called {\rm exponentially bounded}
if there exist $\om\in\R$ and $M=M(\om)\ge 1$ such that
$$
\|U(t,\tau)\|_{\LL(X)}\le Me^{\om(t-\tau)} \mbox{ for all } t\ge\tau.
$$
An evolution family is
{\rm uniformly exponentially stable}
if  $\om<0$.
\end{defn}

We are now prepared to state our first result.

\begin{thm}\label{evol}
Suppose that the coefficients in the system (\ref{ss10}) fulfill the
conditions (\ref{ss44}), (\ref{ss4}), and (\ref{ss8}). Then the problem (\ref{ss10}), (\ref{ss5})
generates an exponentially bounded evolution 
family $\{U(t,\tau)\}_{t\ge\tau}$ on the space $L^2(0,1)^n$. 
 \end{thm} 

Our second result, Theorem \ref{main_smoothing} below, states that the evolution family
has a smoothing property of the following kind.

\begin{defn}\label{defn:smoothing_abstr}  
Let $Y\hookrightarrow Z$ be continuously
embedded Banach spaces and, for each $\tau$ and $t\ge\tau$,  $V(t,\tau)\in\LL(Z)$. 
The two-parameter family
$\{V(t,\tau)\}_{t\ge\tau}$ 
is called {\rm  smoothing} from $Z$ to $Y$ if there is $T>0$ such that 
 $V(t,\tau)\in \LL(Z,Y)$ for all $t\ge\tau+T$.
\end{defn}

From  \reff{rep1} and (\ref{Q})  it follows that, if the system  (\ref{ss10}) is decoupled
(non-diagonal elements of $b$ vanish), then the continuous function $u$ fulfilling the equation
$$
u(x,t)=(Qu)(x,t)
$$ 
is a continuous solution to the  problem (\ref{ss10}), (\ref{ss5}), (\ref{ss11}). 
Let us introduce  a linear bounded operator
 $C\in\LL\left(BC(\R)^n,BC(\Pi)^n\right)$ by
$$
 (Cv)_j(x,t)=c_j(x_j,x,t)v_j(\om_j(x_j,x,t)), \quad j\le n,
$$
where $x_j$ is given by \reff{*k}.
Therefore, every  continuous solution to the decoupled problem
stabilizes to zero in a finite time  $d>0$ if and only if
\beq\label{U-smooth}
\begin{array}{ll}
\mbox{ there is }  
k\in\N \mbox{ such that } (CP)^ku\equiv 0
\mbox{ for all } u\in C\left(\overline\Pi\right)^n.
\end{array}
\ee
Moreover, due to (\ref{ss4}),  the time of extinction  does not exceed  
$\frac{k}{\Lambda_0}$. The condition \reff{U-smooth} will be crucial for the following  
smoothing result, that will be proved in 
Section~\ref{sec:smoothing}.

\begin{thm}\label{main_smoothing}
Assume that the  coefficients $a_j,b_{jk}$, and $p_{jk}$ in  (\ref{ss10}) and (\ref{ss5}) fulfill the
conditions  (\ref{ss44}), (\ref{ss4}), (\ref{ss8}), and \reff{U-smooth}. 
Moreover, assume that 
\beq
\label{cass}
\mbox{for all } 1 \le j \not= k \le n \mbox{ there exists }   \beta_{jk} \in BC^1(\Pi) \mbox{ such that } 
b_{jk}=\beta_{jk}(a_k-a_j).
\ee
Then the evolution family $\{U(t,\tau)\}_{t\ge \tau}$ 
generated by  (\ref{ss10}), (\ref{ss5}) is smoothing 
from $L^2\left(0,1\right)^n$ to  $C^1([0,1])^n$. 
\end{thm}

 \begin{cor} \label{sm1}
Assume that the unperturbed problem (\ref{ss2}), (\ref{ss5}) fulfills the conditions
 (\ref{ss44}), (\ref{ss4}),  and \reff{U-smooth}. 
If 
$\tilde{b}_{jj}\equiv 0$ and  (\ref{cass})
is fulfilled with $\tilde b_{jk}$ in place of $b_{jk}$,
 then the evolution family $\{\tilde U(t,\tau)\}_{t\ge\tau}$
generated by the perturbed problem (\ref{ss7}), (\ref{ss5}) is  smoothing from
 $L^2\left(0,1\right)^n$ to  $C^1([0,1])^n$. 
\end{cor}

The proof of the corollary straightforwardly follows from Theorem \ref{main_smoothing} and the fact that
the condition \reff{U-smooth} is stable with respect to perturbations $\tilde{b}$ of the matrix $b$ such that
$\tilde{b}_{jj}\equiv 0$. 

In  Section \ref{sec:perturb}  we prove the main result stating that the
evolution family of  the perturbed problem (\ref{ss7}), (\ref{ss5})
is exponentially stable.

 \begin{thm}\label{hyper}
Under   the conditions (\ref{ss44}), (\ref{ss4}),  \reff{U-smooth} the following is true.

$(\io)$ For any $\ga>0$ there exist $\eps>0$ and $M=M(\ga)\ge 1$ such that, whenever
 $\max_{j,k}\bigl\|\tilde b_{jk}\bigr\|_\infty<\eps$,
the evolution family $\{\tilde U(t,\tau)\}_{t\ge\tau}$ generated by the perturbed problem
 (\ref{ss7}), (\ref{ss5}) fulfills the bound
\beq\label{bound1}
\left\|\tilde U(t,\tau)\right\|_{\LL\left(L^2(0,1)^n\right)}\le 
M e^{-\ga(t-\tau)} \,\, \mbox{ for }\,\,  t\ge\tau.
\ee

$(\io\io)$ If 
$\tilde{b}_{jj}\equiv 0$ and  (\ref{cass})
is fulfilled  with $\tilde b_{jk}$ in place of $b_{jk}$, then  for any $\ga>0$ there exist $\eps>0$ and $M_1=M_1(\ga)\ge M$ such that,
whenever  $\max_{j,k}\bigl\|\tilde b_{jk}\bigr\|_1<\eps$,
the evolution family $\{\tilde U(t,\tau)\}_{t\ge\tau}$
fulfills the bound
\beq\label{bound2}
\left\|\tilde U(t,\tau)\right\|_{\LL\left(L^2(0,1)^n,C^1([0,1])^n\right)}\le 
M_1 e^{-\ga(t-\tau)}\,\, \mbox{ for }\,\, t\ge\tau+T_1
\ee 
for some constant $T_1>0$.
\end{thm} 

\begin{rem}\rm
The assumption $\tilde{b}_{jj}\equiv 0$ of Theorem \ref{hyper} $(\io\io)$  can be dropped in some cases
(see Examples \ref{exx1} and \ref{exx2}). However,
 we cannot avoid it in general (see Example \ref{bjjne0}). 
The reason is that  the condition  \reff{U-smooth}, 
ensuring the smoothing property for the perturbed problem, can be destroyed by perturbations  
$\tilde{b}_{jj}$ of ${b}_{jj}$.
\end{rem}

\section{Examples and comments}\label{sec:exam}
\renewcommand{\theequation}{{\thesection}.\arabic{equation}}
\setcounter{equation}{0}

\subsection{Condition \reff{cass} is essential}

The condition \reff{cass} is a kind of Levy condition for compensating
weak 
hyperbolicity. Let us show that it is crucial for the regularity result stated in
Theorem~\ref{main_smoothing}.

\begin{ex}\rm
Let $\vphi: \R\to\R$ be a continuous $1$-periodic function which is continuously differentiable on 
$(0,1)$. Consider the following problem in $\Pi_0$:
\begin{equation}
 \begin{array}{ll}
 \displaystyle\partial_{t}u_1
 +\partial_{x}u_1=0,   \\ 
\displaystyle\partial_{t}u_2
 +\partial_{x}u_2
 -u_1=0, 
 \end{array}
 \label{f1ex1}
 \end{equation} 
 \begin{equation}\label{f31}
 \begin{array}{ll}
 u_{1}(x,0)=\vphi(x), \\ 
u_{2}(x,0)=x\vphi(x),
 \end{array}
 \end{equation}
 \beq\label{f2ex1}
 \begin{array}{ll}
  u_{1}(0,t)=u_{2}(1,t), \\ 
  u_{2}(0,t)=0. 
  \end{array}
  \ee
This problem is a particular case of  (\ref{ss10}), (\ref{ss11}), (\ref{ss5})
 and satisfies all assumptions of Theorem~\ref{main_smoothing} except  \reff{cass}. 
It is straightforward to check that
$$
u_1=\vphi(x-t), \;\;\; 
u_2=x\,\vphi(x-t)
$$
is a continuous  solution to the problem \reff{f1ex1}, \reff{f31},  \reff{f2ex1}.
One can easily see that this solution is not continuously differentiable even if $t$
is supposed to be large,
since its  regularity   does not exceed  the regularity 
 of the initial function $\vphi(x)$ for any  $t$.
Thus, the conclusion of Theorem \ref{main_smoothing} cannot be ensured without \reff{cass}.
\end{ex}

\subsection{Theorem \ref{hyper} is not true under ``large''
 perturbations}

Here we present  simple examples showing that if the entries of the 
perturbation matrix $\tilde b$ are not small enough,  then
the statement of Theorem \ref{hyper}  about the exponential stability of
 (\ref{ss7}), \reff{ss11}, (\ref{ss5})  fails.  

\begin{ex}\rm
In $\Pi_0$, let us consider the $2\times 2$-system
\beq
\begin{array}{rcl}\label{ex11}
\d_tu_1+\d_xu_1&=&\nu u_2,\\
\d_tu_2-\d_xu_2&=&0
\end{array}
\ee
with the boundary conditions
\beq\label{ex2}
u_1(0,t)=0,\quad u_2(1,t)=u_1(1,t).
\ee
Note that the unperturbed problem (when $\nu$=0) is superstable. We now show that, if
$\nu>1$, then  \reff{ex11},  \reff{ex2} is  not  exponentially stable.
To this end,
 consider the corresponding spectral problem 
\beq\label{ex3}
\begin{array}{ll}
\la v_1+v_1^\prime=\nu v_2,\,\, \la
 v_2-v_2^\prime =0, \quad 0< x< 1,\\
v_1(0)=0,\quad v_2(1)=v_1(1),
\end{array}
\ee
with  spectral parameter $\la$.
All solutions are given by the formulas
$$
v_1(x,\la)=Ce^{-\la x}+C_1\frac{\nu}{2\la}e^{\la x},\quad v_2(x,\la)=C_1e^{\la x}.
$$
The boundary conditions imply that the problem \reff{ex3} has a nonzero solution if and only if the
spectral parameter $\la$ satisfies the characteristic equation
$$
\frac{\nu}{2\la}(e^\la-e^{-\la})=e^\la,\quad \la\ne 0.
$$
Setting $\ga=2\la$, we get 
\beq\label{ii}
\nu\left(1-e^{-\ga}\right)=\ga.
\ee
If  $\nu>1$, then
 this equation has a positive solution $\ga$ and, hence 
 the spectral problem \reff{ex3} has a positive eigenvalue.
Due to the spectral mapping
 theorem, the problem \reff{ex11}, \reff{ex2} for $\nu>1$ is not  exponentially stable.

In \cite{els} it is proved that, if $|\nu|<1$, then all non-zero solutions to \reff{ii}
have  negative real parts uniformly separated from zero. This means that  the problem \reff{ex11}, \reff{ex2}
is exponentially stable for  $|\nu|<1$.

\end{ex}

\begin{ex}\rm
Now in $\Pi_0$  we consider a $2\times 2$-system with a diagonal lower-order part, namely
\beq
\label{example}
\begin{array}{l}
\d_tu_1+\d_xu_1+\mu u_1-\nu u_2=0,\,\,\d_tu_2-\d_xu_2=0,\quad x \in (0,1),\\
u_1(0,t)=0 , \quad u_1(1,t)=u_2(1,t),
\end{array}
\ee
where $\mu$ and $\nu$ are positive constants.
The unperturbed problem ($\nu=0$) is superstable.
Our aim  is to show that the smallness of  $\nu$ required to ensure the exponential stability 
of the problem \reff{example} can be expressed in terms of the diagonal lower-order coefficient $\mu$.

The corresponding eigenvalue problem reads
\beq
\label{exampleevp}
\begin{array}{l}
v_1^\prime+(\mu+\la) v_1-\nu v_2=0,\,\,v_2^\prime-\la v_2=0,\quad x \in (0,1),\\
v_1(0)=0,\,\,v_1(1)-v_2(1)=0,
\end{array}
\ee
where $\la$ is a spectral parameter.
The system  \reff{exampleevp} is equivalent to
$$
\displaystyle v_1(x)=\frac{c\nu}{2\la+\mu}\left(e^{\la x}-e^{-(\mu+\la)x}\right),\; v_2(x)=ce^{\la x}, 
$$
where
$$
\nu e^{-\mu-2\la}=\nu-2\la-\mu.
$$
Here $c=v_2(0)$ is a nonzero complex constant.
Setting $z=2\la$, we come to the characteristic quasipolynomial equation for $z$, namely
\beq\label{quasi}
z+a+be^{-z}=0,
\ee
where
$
a=\mu-\nu$, $b=\nu e^{-\mu}.
$
This equation is in detail analyzed in \cite{els} in the context of  asymptotic stability 
for ordinary differential equations with retarded arguments. Accordingly to the results obtained 
in \cite{els}, if 
\beq\label{k1k}
\nu<\frac{\mu}{1+e^{-\mu}},
\ee  
then the solutions to  \reff{quasi} have negative real parts,  uniformly separated from zero. Hence, the problem 
\reff{example} is exponentially stable. Moreover, if 
\beq\label{k2k}
\nu>\frac{\mu+1}{1-e^{-\mu}},
\ee 
then at least one solution to \reff{quasi} has a positive real part, which implies that the problem 
\reff{example} for such $\mu$ and $\nu$ is not exponentially stable.

Note that the condition of smallness similar to \reff{k1k} will  naturally appear 
in the proof of 
Theorem \ref{hyper}; cf. the equality \reff{*2}. 
\end{ex}

\subsection{Condition \reff{U-smooth} appears in 
applications}

It is worth to note 
that the condition \reff{U-smooth} is fulfilled  in many applications. 

\begin{ex}\label{exx1}\rm
The papers \cite{Zel,Zel1} discuss catalytic processes in a chemical reactor. 
A chemical reaction is of zero order if the reaction rate does not depend
on the amount of  reactants. Such reactions
are described by the following boundary value
problem for a
$2\times 2$-semilinear hyperbolic system in $\Pi_0$:
\beq\label{ex1}
\begin{array}{rcl}
\displaystyle
\beta \d_t\Theta+\d_x\Theta&=&Q K e^\Theta-\mu (\Theta-\Theta_r),
\\
\displaystyle
\d_t{\Theta_r}-\d_x{\Theta_r}&=&\mu (\Theta-\Theta_r),\\ [2mm]
\Theta(0,t)&=&\Theta_r(0,t), \\ \Theta_r(1,t)&=&\vartheta_0,
\end{array}
\ee
where $\Theta$ denotes the temperature in the reactor and
$\Theta_r$  the temperature in the  refrigerator. 
Moreover, $\be$, $\mu$,  $K$,   $Q$, and $\vartheta_0$ are  positive constants characterizing
a catalyst and a reactant.
It is supposed that the initial data $\Theta(x,0)$ and $\Theta_r(x,0)$ are given.

 Linearization of the problem \reff{ex1}
at the stationary solution $\Theta_0(x)$,  ${\Theta_r}_0(x)$,
whose existence is proved in  \cite{Rom,Zel}, is a boundary value problem 
with respect to 
$u=\Theta-\Theta_0$ and  $v=\Theta_r-{\Theta_r}_0$, namely  the system
\beq\label{eq:sys104}
\begin{array}{rcl}
\displaystyle
\beta \d_tu+\d_xu&=&(Q K e^{\Theta_0}-\mu )u+\mu v,
\\
\displaystyle
\d_tv-\d_xv&=&\mu (u-v),
\end{array}
\ee
subjected to the boundary conditions
\beq\label{eq:sys105}
u(0,t)=v(0,t), \quad v(1,t)=0 
\ee
and the initial conditions
$
u(x,0)=u_0(x)$, $v(x,0)=v_0(x).
$

Note that the system (\ref{eq:sys104}), (\ref{eq:sys105}) is a perturbation
of the following superstable system:
$$
\begin{array}{rcl}
\displaystyle
\beta \d_tu+\d_xu&=&Q K e^{\Theta_0}u,
\nonumber\\
\displaystyle
\d_tv-\d_xv&=&0,
\\ [2mm]
u(0,t)&=&v(0,t),\\ v(1,t)&=&0. 
\end{array}
$$
The condition
\reff{U-smooth} is here fulfilled with $k=2$, that is,  $(CP)^2z=0$ for all 
$z=(u,v)\in C(\overline\Pi)^2$. Indeed,
\beq\label{CP2}
\begin{array}{rcl}
(CPz)_1(x,t)&=&\displaystyle e^{\int_0^x b_{11}(\eta)\,d\eta }v(0,t-\beta x),
\\ [2mm]
(CPz)_2(x,t)&=&0,\\ [2mm]
\displaystyle
((CP)^2z)_1(x,t)&=&\displaystyle e^{\int_0^x b_{11}(\eta)\,d\eta +\int_0^1b_{22}(\eta)\,d\eta }\,v(1,t-\beta x-1)=0,
\end{array}
\ee
where $b_{11}(x)=Q K e^{\Theta_0(x)}$ and
$b_{22}(x)=0$.

From Theorem \ref{hyper} $(\io)$ it follows that, for sufficiently small
$\mu$,  the evolution family  $U(t,\tau)$ generated by the linear problem
(\ref{eq:sys104}), (\ref{eq:sys105}) is uniformly exponentially stable in $L^2$,
with an exponential decay rate $\ga>0$. This means that the eigenvalues  of the corresponding 
eigenvalue problem have negative real parts uniformly separated from zero.
By the linearization principle for autonomous non-linear strictly 
hyperbolic systems in $C^1$, proved in
 \cite{Elt}, the corresponding $C^1$-stationary solution to the nonlinear problem \reff{ex1}
 is exponentially stable in $C^1$, with an exponential decay rate 
$\ga^\prime$ such that
$0<\ga^\prime<\ga$.

On the other side, from  the proof of Theorem \ref{hyper} $(\io\io)$ one can easily see the following: 
 For sufficiently small
$\mu$, the problem (\ref{eq:sys104}), (\ref{eq:sys105}) is uniformly exponentially stable not only 
in $L^2$ but also in $C^1$, despite
$\tilde{b}_{11}=\tilde{b}_{22}=\mu\ne 0$ (while in Theorem \ref{hyper} $(\io\io)$ 
we have   $\tilde{b}_{11}=\tilde{b}_{22}\equiv 0$). 
Our argument works out because  the  
condition \reff{U-smooth}, which entails the smoothing property of the perturbed problem, is 
stable with respect to perturbations of $b_{jj}$, which entails the smoothing property of the perturbed operator
(in general this is not true, see Example~\ref{bjjne0} below).
Indeed, the condition
\reff{U-smooth} for the perturbed problem (\ref{eq:sys104}), (\ref{eq:sys105}) is fulfilled with $k=2$, 
and the equalities   \reff{CP2} are true  with  $b_{11}(x)=Q K e^{\Theta_0(x)}-\mu$ and
$b_{22}(x)=- \mu$.

\end{ex}

\begin{ex}\label{exx2}\rm
A chemical reaction is of first order if the reaction rate  depends linearly
on the concentration of    reactants. In the presence of a catalyst and
 the internal heat exchange, first order reactions are described by the following 
initial-boundary value
problem in $\Pi_0$ for a $3\times 3$-semilinear hyperbolic system:
\beq\label{ex2chem}
\begin{array}{rcl}
\beta \d_t\Theta+\d_x\Theta&=&Q K e^\Theta(1-C)-\mu (\Theta-\Theta_r),
\\
\d_tC+\d_xC&=&K(1-C)e^\Theta,
\\
\d_t{\Theta_r}-\d_x{\Theta_r}&=&\mu (\Theta-\Theta_r),
\\[3mm]
\Theta(0,t)&=&\Theta_r(0,t),\\
 C(0,t)&=&0,\\
 \Theta_r(1,t)&=&\vartheta_0,
\end{array}
\ee
where
$C$ is the concentration of the reactant. The initial data 
$\Theta(x,0)$, $C(x,0)$, and $\Theta_r(x,0)$ are supposed to be given.

In \cite{Rom,Zel} it is proved that  the problem has a stationary solution for certain parameters.
Note that the linearization of \reff{ex2chem} in a neighborhood of the stationary 
solution $\Theta_0(x), C_0(x), {\Theta_r}_0(x)$ is a particular case of our problem. Indeed,
with respect to 
$u=\Theta-\Theta_0$,  $v=C-C_0$, and  $w=\Theta_r-{\Theta_r}_0$, it is represented by the system 
\beq\label{eq:sys102}
\begin{array}{rcl}
\displaystyle
\beta \d_tu+\d_xu&=&(Q K e^{\Theta_0}(1-C_0)-\mu)u-Q K e^{\Theta_0}v+\mu w,
\nonumber\\
\displaystyle
\d_tv+\d_xv&=& K e^{\Theta_0}(1-C_0)u-Ke^{\Theta_0} v,
\nonumber\\
\displaystyle
\d_tw-\d_xw&=&\mu (u-w)
\nonumber
\end{array}
\ee
with the boundary conditions
\beq\label{eq:sys101}
u(0,t)=w(0,t),\quad v(0,t)=0, \quad w(1,t)=0 
\ee
and the initial conditions
$
u(x,0)=u_0(x),v(x,0)=v_0(x), w(x,0)=w_0(x).
$

Note that the system (\ref{eq:sys102}), (\ref{eq:sys101}) is a perturbation
of a superstable system, namely 
$$
\begin{array}{rcl}
\displaystyle
\beta \d_tu+\d_xu&=&0,
\nonumber\\
\displaystyle
\d_tv+\d_xv&=& 0,
\nonumber\\
\displaystyle
\d_tw-\d_xw&=&0,
\\ [2mm]
u(0,t)&=&w(0,t),\\ v(0,t)&=&0, \\ w(1,t)&=&0. 
\end{array}
$$
Again, the condition
\reff{U-smooth} is fulfilled here with $k=2$. Indeed,
\beq\label{00}
\begin{array}{rcl}
(CPz)_1(x,t)&=&\displaystyle e^{\int_0^x b_{11}(\eta)\,d\eta }w(0,t-\beta x),
\\ [2mm]
(CPz)_2(x,t)&=&(CPz)_3(x,t)=0,\\ [2mm]
\displaystyle
((CP)^2z)_1(x,t)&=&\displaystyle e^{\int_0^xb_{11}(\eta)\,d\eta +\int_0^1b_{33}(\eta)\,d\eta }\,w(1,t-\beta x-1)=0,
\end{array}
\ee
where $z=(u,v,w)$ and $b_{11}(x)=b_{33}(x)\equiv 0$.

Similarly to the previous example, the linear problem \reff{eq:sys102}, \reff{eq:sys101} 
 is uniformly exponentially stable in $L^2$ for all sufficiently small
$\mu$ and $K$. Moreover, the corresponding $C^1$-stationary solution to the nonlinear problem \reff{ex2chem}
 is exponentially stable in $C^1$.

The  condition \reff{U-smooth} is, again, 
stable with respect to perturbations of $b_{jj}$. This follows from the formulas \reff{00} where for the perturbed problem
\reff{eq:sys102}, \reff{eq:sys101}  we have
$b_{11}(x)=Q K e^{\Theta_0(x)}(1-C_0(x))-\mu$ and
$b_{33}(x)=- \mu$. Using now the same argument as in Example \ref{exx1}, we conclude that 
 for all sufficiently small
$\mu$ and $K$ the problem \reff{eq:sys102}, \reff{eq:sys101}  is uniformly exponentially stable also in $C^1$.

\end{ex}

\begin{ex}\label{exx3}\rm
Another example is given by the following nonlinear problem in $\Pi_0$ that
comes from the boundary control theory \cite{lakra}:
\begin{eqnarray*}
\d_tu+a_1 \d_xu&=&b(x)u+c(x)v+uv,\\
\d_tv-a_2\d_xv&=&(1-b(x))u+(1-c(x))v-uv,\\[3mm]
u(0,t)&=&0,\\
v(1,t)&=&h(u(1,t),\la,u_d),
\end{eqnarray*}
where
$a_1>0$, $a_2>0$,
$\la$ is a control parameter   based on the boundary measurement
$u(1,t)$, and
 $u_d$ denotes the desired level required for the signal output $u(1,t)$.
Again, the linearization is covered by  \reff{ss10}, \reff{ss5} and fulfills \reff{U-smooth}
with $k=2$.
\end{ex}

\subsection{Our results apply to nonlinear problems}
Note that in Examples \ref{exx1}--\ref{exx3} above we deal with nonlinear problems.

\subsection{Condition \reff{U-smooth}  is in general not stable with respect to perturbations of $b_{jj}$} 
\begin{ex}\label{bjjne0}\rm
Let us consider the superstable system  perturbed in the diagonal lower order part, namely
\beq\label{non1}
\partial_tu_1+\partial_xu_1=0,\quad \partial_tu_2+\partial_xu_2=\eps u_2,\quad  \partial_tu_3-\partial_xu_3=0,
\ee
and supplement it with the reflection boundary conditions
\beq\label{ii2}
u_1(0,t)=u_3(0,t),\quad u_2(0,t)=u_3(0,t),\quad u_3(1,t)=u_1(1,t)-u_2(1,t).
\ee
One can easily check that for the unperturbed problem ($\eps=0$)  the condition \reff{U-smooth} is fulfilled with $k=3$,
while for the perturbed problem ($\eps\ne 0$) it is not fulfilled for any $k\in\N$. This follows from the following simple
calculations:
$$
\begin{array}{rcl}
(CPz)_1(x,t)&=& u_3 (0,t-x),
\\ (CPz)_2(x,t)&=& e^{\varepsilon x} (CPz)_1(x,t),\\
(CPz)_3(x,t)&=&u_1(1,t+x-1)-u_2(1,t+x-1),
\\
((CP)^2z)_1(x,t)&=& u_1(1,t-x-1)-u_2(1,t-x-1), 
\\ ((CP)^2z)_2(x,t)&=&e^{\varepsilon x} ((CP)^2z)_1(x,t),
\\ 
((CP)^2z)_3(x,t)&=&(1-e^{\varepsilon})u_3(0,t+x-2),
\\
((CP)^3z)_1(x,t)&=&(1-e^{\varepsilon})u_3(0,t-x-2), \\ ((CP)^3z)_2(x,t)&=&e^{\varepsilon x} ((CP)^3z)_1(x,t),
\\
((CP)^3z)_3(x,t)&=&(1-e^{\varepsilon})(u_1(1,t+x-3)-u_2(1,t+x-3)),
\end{array}
$$
where $z=(u_1,u_2,u_3)$.

\end{ex}

\section{Evolution families}\label{sec:abstr}
\renewcommand{\theequation}{{\thesection}.\arabic{equation}}
\setcounter{equation}{0}

In this section for the problem \reff{ss10}, (\ref{ss11}), (\ref{ss5})
 we 
prove the existence of an evolution family on 
$L^2(0,1)^n$  and show that it is eventually differentiable. 
Existence of evolution families 
 on the spaces of continuous  functions
is a complicated question because one has to take into account compatibility conditions 
between initial and boundary data.  The  point is that 
the first order compatibility conditions
 depend on the coefficients of the differential equations and 
are not the same for unperturbed and perturbed problems. 
This complication can be avoided (and we will follow this way)
working with
 evolution operators defined on  $L^2(0,1)^n$, where the compatibility conditions
do not play any role. Nevertheless, the main result (Theorem \ref{hyper} 
stating the exponential stability)
will be proved in both  $L^2(0,1)^n$ and  $C^1\left([0,1]\right)^n$-spaces. For the latter we
will use the result in $L^2(0,1)^n$ and  the smoothing property of the evolution families 
in the sense of Definition \ref{defn:smoothing_abstr}.

\subsection{A priori estimates}\label{sec:exist}

The following results about existence and uniqueness of continuous and classical solutions 
to the problem 
(\ref{ss10}), (\ref{ss11}), (\ref{ss5}) readily follow from
\cite{ijdsde,Km1}. 

\begin{thm}\cite{ijdsde,Km1}\label{km}
Suppose that the coefficients $a_j$ and $b_{jj}$  fulfill the 
  conditions  \reff{ss44}, \reff{ss4}, and \reff{ss8}. Let $\tau \in \R$ be 
 arbitrary fixed. 
If  $\varphi\in C^1([0,1])^n$ fulfills the zero order compatibility conditions  
\beq\label{eq:nl1}
\begin{array}{l}
\displaystyle
\varphi_j(0) = 
\sum\limits_{k=1}^mp_{jk}\varphi_k(1)+\sum\limits_{k=m+1}^np_{jk}\varphi_k(0),
\quad  1\le j\le m,
\nonumber\\
\displaystyle
\varphi_j(1) = \sum\limits_{k=1}^mp_{jk}\varphi_k(1)+\sum\limits_{k=m+1}^np_{jk}\varphi_k(0),
\quad   m< j\le n,
\nonumber
\end{array}
\ee
 and the 
first order compatibility conditions 
\beq\label{eq:nl2}
\begin{array}{l}
\displaystyle
\psi_j(0) = 
\sum\limits_{k=1}^mp_{jk}\psi_k(1)+\sum\limits_{k=m+1}^np_{jk}\psi_k(0),
\quad  1\le j\le m,
\nonumber\\
\displaystyle
\psi_j(1) = \sum\limits_{k=1}^mp_{jk}\psi_k(1)+ \sum\limits_{k=m+1}^np_{jk}\psi_k(0),
\quad   m< j\le n,
\nonumber
\end{array}
\ee
where
$$
\psi(x)= -(a(x,\tau)\partial_x + b(x,\tau))\varphi(x),
$$
then in $\Pi_\tau$ there exists a unique classical (continuously differentiable) solution $u(x,t)$ to the problem (\ref{ss10}), (\ref{ss11}), (\ref{ss5}). Moreover,  there are 
 constants $K_1\ge 1$ and $\om_1>0$ not depending on $\tau$, $t$, and $\varphi$ such that
\beq\label{eq:apr5}
\|u(\cdot,t)\|_{C^1([0,1])^n}\le K_1e^{\omega_1(t-\tau)}\|\vphi\|_{C^1([0,1])^n}\,\,\mbox{ for }\,\,
 t\ge\tau.
\ee
Given $c>0$, the constants $K_1$ and $\omega_1$ can be chosen the same for all 
$b_{jk}$ such that $\max_{j,k}\bigl\|b_{jk}\bigr\|_1<c$.
\end{thm}

\begin{lem}\label{nn0}  If  $u(x,t)$ is a classical solution 
to the problem  (\ref{ss10}), (\ref{ss11}), (\ref{ss5}), then it fulfills the estimate
\beq\label{eq:apr2}
\|u(\cdot,t)\|_{L^2(0,1)^n}\le K_2e^{\omega_2(t-\tau)}\|\vphi\|_{L^2(0,1)^n}\,\,\mbox{ for }\,\,
 t\ge\tau
\ee
with some constants  $K_2\ge 1$ and $\omega_2>0$  not depending on  $\tau$, $t$, and $\varphi$. 
\end{lem}
\begin{proof} The proof is based on the argument from \cite{god}
used to get a priori estimates
for initial-boundary value problems for
first order hyperbolic systems, now for decoupled boundary conditions. 
Let $u=u(x,t)$  be a classical solution to the problem under consideration.  
Take a scalar product of \reff{ss10} with $u$ and integrate the resulting system over the rectangle
$\Pi_\tau^t=\{(x,\theta)\,:\,0 < x < 1, \tau < \theta < t\}$. We get
$$
\int\int_{\Pi_\tau^t}\left(\frac{\partial }{\partial \theta}(u,u)+\frac{\partial }{\partial x}(au,u)\right)\,\, dx d\theta=
\int\int_{\Pi_\tau^t}\left(-2(bu,u)+(a_x u,u)\right)\, dx d\theta,
$$
where  $(\cdot,\cdot)$ denotes 
the scalar
product
 in $\R^n$.
  Applying  Green's formula to the left hand side, we obtain  
\beq\label{god04}
\begin{array}{cc}
\displaystyle\|u(\cdot,t)\|_{L^2(0,1)^n}^2+\int_\tau^t \left(\sum_{j=1}^n a_j(1,\theta)
 u_j^2(1,\theta)-\sum_{j=1}^n 
a_j(0,\theta) u_j^2(0,\theta)\right)\,d\theta\\
\displaystyle=\|\vphi\|^2_{L^2(0,1)^n}+\int\int_{\Pi_\tau^t}\left(-2(bu,u)+(a_x u,u)\right)\, 
dxd\theta.
\end{array}
\ee

Suppose first that the boundary  
conditions (\ref{ss5}) are dissipative, i.e.
\beq\label{god1}
\sum\limits_{j=1}^ma_j(1,t)u_j^2(1,t)-\sum\limits_{j=m+1}^na_j(0,t)u_j^2(0,t)+\sum\limits_{j=m+1}^na_j(1,t)(Pu)_j(t)^2-\sum\limits_{j=1}^ma_j(0,t)(Pu)_j(t)^2\ge 0.
\ee 
Then from  \reff{god04} we have
\begin{eqnarray*}
\displaystyle\|u(\cdot,t)\|_{L^2(0,1)^n}^2&\le& 
\displaystyle\|\vphi\|^2_{L^2(0,1)^n}+\int\int_{\Pi_\tau^t}\left|\left((a_x-2b)u,u\right)\right|\,dx d\theta\\
&\le& \displaystyle\|\vphi\|^2_{L^2(0,1)^n}+
\be\int_\tau^t I(\theta)\, d\theta,
\end{eqnarray*}
where $I(t)=\|u(\cdot, t)\|^2_{L^2(0,1)^n}=\int_0^1 (u,u)\, dx$, $\be=n\max_{i,j} \|(a_x-2b)_{ij}\|_\infty$,
and $(a_x-2b)_{ij}\in BC(\Pi)$ are entries of the matrix $a_x-2b$. 
Applying  Gronwall's argument to the inequality 
$$
I(t)\le \|\vphi\|^2_{L^2(0,1)^n}+ \be \int_\tau^t I(\theta)\,d\theta,
$$ 
we come to the estimate 
 $\|u(\cdot,t)\|_{L^2(0,1)^n}\le e^{\omega_2(t-\tau)}\|\vphi\|_{L^2(0,1)^n}$
with  constant  $\omega_2=\frac{\be}{2}$ depending on the coefficients
of the system (\ref{ss10}), (\ref{ss5}) but not on $\vphi$. 

To complete the proof, it remains to show that the inequality \reff{god1}, 
supposed above, causes no loss of generality.
Let  $\mu_i(x,t)$ be arbitrary smooth functions
satisfying the conditions
$$
\inf\limits_{\overline\Pi_\tau}|\mu_j|>0, \quad \sup\limits_{\overline\Pi_\tau}|\mu_j|<\infty
\quad\mbox{ for all } j\le n.
$$
The change of each variable 
$u_j$ to $v_j=\mu_ju_j$
brings the  system (\ref{ss10})  to
\begin{equation}\label{god2}
\partial_tv_j  + a_j(x,t)\partial_xv_j - \frac{\d_t\mu_j+a_j(x,t)\d_x\mu_j}{\mu_j}v_j + 
\sum\limits_{k=1}^nb_{jk}\frac{\mu_j}{\mu_k}v_k = 0
\end{equation}
and the boundary conditions (\ref{ss5}) to
\beq\label{god3}
\begin{array}{l}
\displaystyle
v_j(0,t) = 
\sum\limits_{k=1}^mp_{jk}\frac{\mu_j(0,t)}{\mu_k(1,t)}v_k(1,t)+
\sum\limits_{k=m+1}^np_{jk}\frac{\mu_j(0,t)}{\mu_k(0,t)}v_k(0,t),
\quad 1\le j\le m,
\nonumber\\
\displaystyle
v_j(1,t) =\sum\limits_{k=1}^mp_{jk}\frac{\mu_j(1,t)}{\mu_k(1,t)}v_k(1,t)+
 \sum\limits_{k=m+1}^np_{jk}\frac{\mu_j(1,t)}{\mu_k(0,t)}v_k(0,t),
\quad   m< j\le n.
\nonumber
\end{array}
\ee
Note that the  resulting system
 \reff{god2}, \reff{god3} is of the type  (\ref{ss10}), (\ref{ss5}), and  the  
inequality \reff{god1} for it reads
\beq\label{god4}
\begin{array}{ll}
\displaystyle
\sum\limits_{j=1}^ma_j(1,t)v_j^2(1,t)-\sum\limits_{j=m+1}^na_j(0,t)v_j^2(0,t)\\
+\displaystyle\sum\limits_{j=m+1}^na_j(1,t)
\left[
\sum\limits_{k=m+1}^np_{jk}\frac{\mu_j(1,t)}{\mu_k(0,t)}v_k(0,t)+
\sum\limits_{k=1}^mp_{jk}\frac{\mu_j(1,t)}{\mu_k(1,t)}v_k(1,t)
\right]^2\\
\displaystyle
-\displaystyle\sum\limits_{j=1}^ma_j(0,t)
\left[
\sum\limits_{k=m+1}^np_{jk}\frac{\mu_j(0,t)}{\mu_k(0,t)}v_k(0,t)+
\sum\limits_{k=1}^mp_{jk}\frac{\mu_j(0,t)}{\mu_k(1,t)}v_k(1,t)
\right]^2
\ge 0.
\end{array}
\ee 
One can easily see that the functions  $\mu_j$ can be chosen so that  the left hand side of \reff{god4}  
is a non-negative definite quadratic form with respect to
 $v_j(1,t),$ $j\leq m$ and $v_j(0,t),$ $m+1\leq j\leq n$. Indeed, since $a_j>0$ for $j\le m$ and $a_j<0$ for $m+1\le j\le n$ by the assumption
\reff{ss4},  the first line of \reff{god4} is a positive  definite quadratic form and 
 the last two lines are a non-negative  definite quadratic form. Now we choose the functions  $\mu_j$ so that 
$\mu_j(0,t)$ for $m+1\le j\le n$ and $\mu_j(1,t)$ for $j\le m$ are so large, while
$\mu_j(1,t)$ for $m+1\le j\le n$ and $\mu_j(0,t)$ for $j\le m$ are so small 
that the whole expression in \reff{god4} is a  non-negative  definite quadratic form.

This completes the proof.
\end{proof}

\subsection{Existence of evolution families (proof of Theorem \ref{evol})}

 Theorem  \ref{km} guarantees the existence and uniqueness of the classical solution to the problem
  (\ref{ss10}), (\ref{ss11}), (\ref{ss5})  for any $\tau\in\R$ and any initial function 
$\varphi \in C^\infty_0([0,1])^n$.
Fix  arbitrary $\tau\in\R$, $\vphi\in L^2(0,1)^n$, and a
sequence  $\vphi^l\in C_0^\infty([0,1])^n$
such that   $\vphi^l\to\vphi$ in $L^2(0,1)^n$. Let $u^l$ denote  
the continuously differentiable solutions to the problem
 (\ref{ss10}), (\ref{ss11}), (\ref{ss5}) with $\vphi(x)=\varphi^l(x)$.
The bound (\ref{eq:apr2}) implies that
$$
\max_{\tau \le \theta \le t}\|u^m(\cdot,\theta)-u^l(\cdot,\theta)\|_{L^2(0,1)^n} \to 0
\mbox{ as } \, m,l \to\infty$$
 for each  $t>\tau$.
It follows that, given $t>\tau$, the sequence  $u^l$  converges in $C([\tau,t], L^2(0,1))^n$. 
Moreover,  (\ref{eq:apr2}) ensures that the limit function $u$ 
does not depend on the choice of 
 $\vphi^l$. 

Define $U(t,\tau)\varphi=u(\cdot,t)$. Thus, 
$U(t,\tau)\varphi\in L^2(0,1)^n$ for each $t\geq\tau$ and  $\varphi \in L^2(0,1)^n$. 
Since the classical solution to  (\ref{ss10}), (\ref{ss11}), (\ref{ss5})  is unique, the family
 $\{U(t,\tau)\}_{t\ge\tau}$  fulfills the first property  in Definition \ref{def_evolfam}. 
Moreover, the estimate  (\ref{eq:apr2}) entails the 
exponential bound
 $$
\|U(t,\tau)\|_{\LL(L^2([0,1])^n)}\le K_2e^{\omega_2(t-\tau)},\quad t\ge\tau. 
$$
as well as the second property in Definition \ref{def_evolfam}.
More specifically,   the strong continuity of $U(t,\tau)$ in $t$ (and similarly in $\tau$)
follows from the convergence
\begin{eqnarray*}
\|U(t,\tau)\vphi-U(\tilde t,\tau)\vphi\|_{L^2(0,1)^n} = \|(U(t,\tilde t)-I)U(\tilde t,\tau)\vphi\|_{L^2(0,1)^n}  \to 0
\mbox{ as } \, t \to\tilde t \quad(t\ge\tilde t\ge\tau)
\end{eqnarray*}
 for all $\vphi\in L^2(0,1)^n$. To prove this convergence, it is sufficient 
to note that $U(\tilde t,\tau)\vphi\in L^2(0,1)^n$ and to show that
\beq\label{21}
\|U(t,\tau)\vphi-\vphi\|_{L^2(0,1)^n} \to 0
\mbox{ as } \, t \to\tau \quad(t\ge\tau)\qquad 
\ee
for all $\vphi\in L^2(0,1)^n$.
To this end, given $\vphi\in L^2(0,1)^n$, take an arbitrary 
sequence  $\vphi^l\in C_0^\infty([0,1])^n$
such that   $\vphi^l\to\vphi$ in $L^2(0,1)^n$ as $l\to\infty$. Then
\begin{eqnarray*}
\lefteqn{
\|U(t,\tau)\vphi-\vphi\|_{L^2(0,1)^n}
}\\
&& \le \|U(t,\tau)\vphi-U(t,\tau)\vphi^l\|_{L^2(0,1)^n}+\|U(t,\tau)\vphi^l-\vphi^l\|_{L^2(0,1)^n}+
\|\vphi^l-\vphi\|_{L^2(0,1)^n}.
\end{eqnarray*}
As $l\to\infty$, the first summand in the right-hand side tends to zero by (\ref{eq:apr2}). The second 
summand tends to zero because 
$U(t,\tau)\vphi^l$ for each $l$ is the classical solution. The third one tends to zero by the choice of 
$\vphi^l$. This 
completes the proof of \reff{21} and, therefore, the strong continuity of $U(t,\tau)$ in $t$, as desired.

Therefore, $\{U(t,\tau)\}_{t\ge\tau}$  determines an exponentially bounded evolution family generated by 
the problem (\ref{ss10}), (\ref{ss11}) in the sense of Definitions \ref{def_evolfam}
and \ref{def_eo}.
Theorem \ref{evol} is therewith proved.

This proof motivates the following  definition
of an  $L^2$-generalized solution to the problem
 (\ref{ss10}), (\ref{ss11}), (\ref{ss5}). 

 \begin{defn}
Given $\vphi\in L^2(0,1)^n$, let  $\vphi^l\in C_0^\infty([0,1])^n$ be an arbitrary fixed sequence 
such that   $\vphi^l\to\vphi$ in $L^2(0,1)^n$.
A function $u\in C\left([\tau,\infty), L^2(0,1)\right)^n$ is called 
an {\rm $L^2$-generalized  solution} to the problem (\ref{ss10}), (\ref{ss11}), (\ref{ss5}) 
if the sequence 
of  continuously differentiable solutions $u^l(x,t)$ to the problem
(\ref{ss10}), (\ref{ss11}), (\ref{ss5}) with
 $\vphi(x)=\vphi^l(x)$  fulfills the convergence
$$
\|u(\cdot,\theta)-u^l(\cdot,\theta)\|_{L^2(0,1)^n} \to_{l\to\infty} 0,
$$
uniformly in $\theta$ varying in the range $\tau\le\theta\le t$, for every $t>\tau$.
\end{defn}

\subsection{Smoothing property}\label{sec:smoothing}

In this section we  show that,
under  the conditions \reff{U-smooth} and (\ref{cass}), the $L^2$-generalized 
solutions to (\ref{ss10}), (\ref{ss11}), (\ref{ss5})
 become continuous in a finite time.
 Furthermore,
the time at which the
solutions 
reach the $C$-regularity does not exceed the value $\tau+T_0$ for  a fixed number $T_0>0$, whatsoever the initial 
time $\tau\in\R$ and the initial function $\vphi\in L^2(0,1)^n$. 
Furthermore,  the function $u(x,t)=[U(t,\tau)\vphi](x)$
satisfies
 the zero order compatibility conditions at points $(0,t)$ and $(1,t)$.  We summarize
 this in the following 
lemma. Let $Y_0$ denote the subspace of $C([0,1])^n$
of functions satisfying the zero-order compatibility conditions \reff{eq:nl1}.

\begin{lem}\label{lem:d}
Suppose that the conditions  (\ref{ss44}), (\ref{ss4}), (\ref{ss8}), \reff{U-smooth}, (\ref{cass})  
 are fulfilled.  Then the evolution family
$\{U(t,\tau)\}_{t\ge\tau}$ generated by the problem
(\ref{ss10}), (\ref{ss5})  is  smoothing  from
  $L^2\left(0,1\right)^n$  to
$Y_0$. 
\end{lem}

The proof develops  the ideas of
\cite{kmit,Km}
where
 the smoothing property is proved from $Y_0$ to $C^1$, and
 it is shown  that the solutions reach the $C^k$-regularity in a finite time for each $k$.
Here we  extend the smoothing results to the case where the initial 
data are $L^2$-functions only.

\begin{proof} 
 From Theorem \ref{evol}  it follows that for all $\tau\in\R$ and $\vphi\in L^2(0,1)^n$ the problem
 (\ref{ss10}), (\ref{ss11}), (\ref{ss5}) has a unique    $L^2$-generalized solution $u$.
Note that
$u\in C\left([\tau, \infty), L^2(0,1)\right)^n$. It suffices to show that 
$u\in C\left(\overline\Pi_{\tau+T_0}\right)^n$ 
 for some $T_0>0$ not depending on $\tau$ and $\varphi$.

Fix an arbitrary $\tau\in\R$ and $\vphi\in L^2(0,1)^n$.
Due to \reff{U-smooth}, we can fix $d>0$ such that 
 $\left((CP)^ku\right)(x,t)= 0$ for all 
$(x,t)\in\overline\Pi_{\tau+d}$.
Fix  an arbitrary sequence of functions $\vphi^l\in C_0^\infty([0,1])^n, \, l\in\N,$
such that   $\vphi^l\to\vphi$ in $L^2(0,1)^n$. Let  $u^l$ denote
the classical solutions to the problem
 (\ref{ss10}), (\ref{ss11}), (\ref{ss5}) with $\vphi(x)=\varphi^l(x)$.
Due to the bound (\ref{eq:apr2}),
$
u^l\to u \mbox{ in}\,\,  C([\tau,\theta], L^2(0,1))^n \mbox{ as } \, l\to\infty
$
for any $\theta>\tau$. In order to prove the lemma, it is sufficient to show that 
$$
u^l \mbox{ converges in }  C\left(\overline\Pi_{\tau+2d}^{\tau+2d+\al}\right)^n
\mbox{ as } l\to\infty,
$$
for any $\al>0$.  Here and below,  given $\be$ and $\ga$ such that $\be<\ga$,  we use the notation 
$
\Pi_\be^\ga=\Pi_{\be}\setminus\overline\Pi_{\ga}
$.

Given $\al>0$,
let a  linear bounded operator 
$D:  C\left(\overline\Pi_\tau^{\tau+2d+\al}\right)^n \mapsto  C\left(\overline\Pi_\tau^{\tau+2d+\al}\right)^n$
 be  defined  by the formula
\begin{eqnarray}\label{D}
\left(Dw\right)_j(x,t) = 
\displaystyle-\int_{x_j(x,t)}^x d_j(\xi,x,t)\sum_{k\not=j}b_{jk}(\xi,\om_j(\xi))
w_k(\xi,\om_j(\xi))\,d\xi,\quad j\le n,
\end{eqnarray}
where the functions  $\om_j(\xi,x,t)$ are given by (\ref{char}). For simplicity, 
here and in what follows we use the notation
 $\om_j(\xi)=\om_j(\xi,x,t)$.  Moreover,  we drop the dependence of
$D$ on $\al$, as throughout the proof
$\al$ is 
arbitrary fixed.

Accordingly to (\ref{rep1}), the solution $u^l\in  C^1\left(\overline\Pi_{\tau}\right)^n$
 satisfies the operator equation
\begin{eqnarray}
\label{tau+d}
u^l=Qu^l+Du^l.
\end{eqnarray}
In particular,
\begin{eqnarray}
\label{rep_u0}
u^l\big|_{\overline\Pi_{\tau+d}}=CPu^l+Du^l.
\end{eqnarray}
 Putting \reff{rep_u0} into the first summand in  \reff{rep_u0},
we get
\begin{eqnarray*}
u^l\big|_{\overline\Pi_{\tau+d}}=(CP)^2u^l+(I+CP)Du^l.
\end{eqnarray*}
Iterating this, that is, substituting \reff{rep_u0} into the last  equation once and once again, 
in the $k$-th step we meet the property \reff{U-smooth}
and get the formula
\begin{eqnarray}\label{rep_u1}
u^l\big|_{\overline\Pi_{\tau+d}}=\sum_{i=0}^{k-1}(CP)^{i}Du^l.
\end{eqnarray}
Since $u^l$ occurs in both sides of \reff{rep_u1}, this equation can be iterated. Note that $D$
operates with $u^l$ on a different (shifted) domain. Hence, such iteration is possible only on a subdomain
of $\overline\Pi_{\tau+d}$. This is possible on $\overline\Pi_{\tau+2d}$ and, doing so, we obtain
\begin{eqnarray}\label{rep_u11}
u^l\big|_{\overline\Pi_{\tau+2d}}=
\sum_{i=0}^{k-1}(CP)^{i}D\sum_{j=0}^{k-1}(CP)^{j}D u^l= \sum_{i=0}^{k-1}(CP)^{i}D\sum_{j=1}^{k-1}(CP)^{j}D u^l+\sum_{i=0}^{k-1}(CP)^{i}D^2 u^l.
\end{eqnarray}
We now have  to prove that  the right-hand side of \reff{rep_u11} converges in 
 $C\left(\overline\Pi_{\tau+2d}^{\tau+2d+\al}\right)^n$  as $l~\to~\infty$ for all $\al>0$.
Fix an arbitrary $\al>0$. It  suffices to show that 
\beq\label{conv_all}
 \mbox{ the sequences }  D(CP)^iDu^l \,\, \mbox{and} \,\, D^2u^l 
 \mbox{ converge in } C\left(\overline\Pi_{\tau+2d}^{\tau+2d+\al}\right)^n
\ee
as  $l\to\infty$ for all  $i=0,1,\dots, k-1$.
The proof of \reff{conv_all} will be divided into two claims.

{\it Claim 1. The sequence $D^2u^l$ converges in 
$C\left(\overline\Pi_{\tau+2d}^{\tau+2d+\al}\right)^n$ as  $l\to\infty$.}
On the account of \reff{D}, after changing the order of integration,
 we derive the following formula for $\left[D^2u^l\right]_j(x,t)$ on
 $\overline\Pi_{\tau+2d}^{\tau+2d+\al}$ for each $j\le n$:
\begin{eqnarray}\label{D11}
\left[D^2u^l\right]_j(x,t)
=\sum_{k\not=j}^n\sum_{i\not=k}^n
\int_{x_j}^x \int_\eta^x d_{jki}(\xi,\eta,x,t)b_{jk}(\xi,\om_j(\xi))
u_i^l(\eta,\om_k(\eta,\xi,\om_j(\xi))) d \xi d \eta,
\end{eqnarray}
where $x_j$ is given by \reff{*k} and 
$$
d_{jki}(\xi,\eta,x,t)
=d_j(\xi,x,t)d_k(\eta,\xi,\om_j(\xi))b_{ki}(\eta,\om_k(\eta,\xi,\om_j(\xi))).
$$
Note that, due to \reff{cass}, given $j\le n$ and $k\ne j$, the function $b_{jk}$ vanishes for those 
$\xi\in[0,1]$ such that 
 $a_k(\xi,\om_j(\xi))=a_j(\xi,\om_j(\xi))$.

Now, for  fixed $k\not=j$ and  $\eta$, let us change the variables 
\beq\label{change}
\xi\mapsto\theta=\om_k(\eta,\xi,\om_j(\xi)).
\ee
Taking into  account
the equalities
$$
\d_x\om_j(\xi,x,t)=-\frac{1}{a_j(x,t)} \exp \int_\xi^x 
\frac{\d_ta_j(\eta,\om_j(\eta))}{a_j(\eta,\om_j(\eta))^2}\, d\eta,
$$
$$
\d_t\om_j(\xi,x,t)= \exp \int_\xi^x 
\frac{\d_ta_j(\eta,\om_j(\eta))}{a_j(\eta,\om_j(\eta))^2}\, d\eta,
$$
from \reff{change} we get
\begin{eqnarray}\label{theta}
d\theta& = & \left[\d_2\om_k(\eta,\xi,\om_j(\xi))+\d_3\om_k(\eta,\xi,\om_j(\xi))\d_\xi\om_j(\xi)\right]d\xi\nonumber\\
\displaystyle
& = &\frac{a_k(\xi,\om_j(\xi))-a_j(\xi,\om_j(\xi))}{a_j(\xi,\om_j(\xi))a_k(\xi,\om_j(\xi))}
\d_3\om_k(\eta,\xi,\om_j(\xi))\, d\xi,
\end{eqnarray}
where $\d_i$ here and in what follows denotes the partial derivative with respect to the $i$-th  argument.
It follows
from \reff{theta} that \reff{change} is non-degenerate  for all $\xi\in[0,1]$
fulfilling the condition
 $a_k(\xi,\om_j(\xi))\ne a_j(\xi,\om_j(\xi))$. Hence, for given $\eta,\theta,x$, and $t$, 
there exists a unique solution $\xi=\tilde x(\theta,\eta,x,t)$ to the equation
$
\omega_k(\xi,\eta,\theta)= \omega_j(\xi,x,t),
 $
 and we have
$$
\om_k(\tilde x(\theta,\eta,x,t),\eta,\theta)=\om_j(\tilde x(\theta,\eta,x,t),x,t).\nonumber
$$
Changing the variables according to \reff{change},
we  obtain
\begin{eqnarray}\label{after_change}
\lefteqn{
\int_{x_j}^x \int_\eta^x d_{jki}(\xi,\eta,x,t)b_{jk}(\xi,\om_j(\xi))
u_i^l(\eta,\om_k(\eta,\xi,\om_j(\xi))) d \xi d \eta}\nonumber\\
&&=\int_{x_j}^x \int_{\om_j(\eta,x,t)}^{\om_k(\eta,x,t)}
d_{jki}(\tilde x,\eta,x,t)\beta_{jk}(\tilde x,\om_j(\tilde x))
\frac{(a_ka_j)(\tilde x,\om_j(\tilde x))}{\d_3\om_k(\eta,\tilde x,\om_j(\tilde x))}
u_i^l(\eta,\theta) d \theta d \eta,
\end{eqnarray}
where $\beta_{jk}$ are continuous functions fulfilling \reff{cass}. 
Note that $\be_{jk}(x,t)$ are not uniquely defined by \reff{cass} for $(x,t)$ such that 
$a_{j}(x,t)=a_{k}(x,t)$. Nevertheless, the
 left-hand side and, hence, the right-hand
 side of \reff{after_change} do not depend on the choice of $\be_{jk}$.
This easily follows
 from \reff{cass} and \reff{theta}, entailing that $b_{jk}(\xi,\om_j(\xi))=0$ and
   $d\theta=0$ if $a_{j}(\xi,\om_j(\xi))=a_{k}(\xi,\om_k(\xi))$.
Changing the order of integration in the right-hand side of  \reff{after_change},
we rewrite it as follows (where for definiteness we suppose that $j,k\le m$ and $a_k<a_j$, hence $\om_j(\xi)<\om_k(\xi)$; the
other cases are treated similarly):
\begin{eqnarray}\label{change_order}
\lefteqn{
 \int_{\om_j(0)}^{\om_k(0)}\int_{0}^{\tilde\om_j(\theta)}
d_{jki}(\tilde x,\eta,x,t)\beta_{jk}(\tilde x,\om_j(\tilde x))
\frac{(a_ka_j)(\tilde x,\om_j(\tilde x))}{\d_3\om_k(\eta,\tilde x,\om_j(\tilde x))}
u_i^l(\eta,\theta) d \eta d \theta
}\nonumber\\
&&+ \int^t_{\om_k(0)}\int^{\tilde\om_j(\theta)}_{\tilde\om_k(\theta)}
d_{jki}(\tilde x,\eta,x,t)\beta_{jk}(\tilde x,\om_j(\tilde x))
\frac{(a_ka_j)(\tilde x,\om_j(\tilde x))}{\d_3\om_k(\eta,\tilde x,\om_j(\tilde x))}
u_i^l(\eta,\theta) d \eta d \theta,
\end{eqnarray}
where 
$ \tilde\om_s(\tau)=\tilde\om_s(\tau,x,t)$
denotes the inverse of the function from $[0,1]$ to $\R$ taking
$\xi$ to $\tau=\om_s(\xi,x,t)$.
Note that the range of integration in $\theta$ in both integrals 
does not exceed~$d$ in length. This follows from the fact that the time needed to reach the boundary
$x=0$ or $x=1$ from any point $(x,t)\in\overline\Pi$ is not larger than $d$.
The $C\left(\overline\Pi_{\tau+2d}^{\tau+2d+\al}\right)^n$-norm
of the function \reff{change_order} can be estimated from above by
\begin{eqnarray}\label{conv1}
&2d\displaystyle
\max\limits_{x,\xi,\eta\in[0,1]}\max\limits_{t,\theta\in[\tau,\tau+2d+\al]}
\left|d_{jki}(\xi,\eta,x,t))\beta_{jk}(\xi,\theta)
\frac{\left(a_ka_j\right)
(\xi,\theta)}{\d_3\om_k(\eta,\xi,\theta)}\right|
\max\limits_{t\in[\tau,\tau+2d+\al]}\int_0^1|u_i^l(\eta,t)|\,d\eta&\nonumber\\
&\displaystyle\le K\left\|u_i^l\right\|_{C([\tau,\tau+2d+\al],L^2(0,1))},&
\end{eqnarray}
where $K>0$ is a constant that  depends on the coefficients $a$ and $b$ but does not
depend on the function
 $u^l$.  
Thus Claim 1 is proved.

{\it Claim 2. The sequence $DCPDu^l$ converges in 
$C\left(\overline\Pi_{\tau+2d}^{\tau+2d+\al}\right)^n$ as  $l\to\infty$.}
We have
\begin{eqnarray}
\lefteqn{
\displaystyle
[DCPDu^l]_j(x,t)=}\nonumber\\
&&=\sum_{k\not=j}^n
\int_x^{x_j}  d_j(\xi,x,t)  b_{jk}(\xi,\om_j(\xi))
c_k\left(x_k,\xi,\om_j(\xi)\right)
\left(PDu^l\right)_k(\om_k(x_k,\xi,\om_j(\xi)))\,d\xi
\nonumber\\
&&\displaystyle
=\sum_{k\not=j}^n\int_{\om_k(x_k)}^{\om_k(x_k,x_j,\om_j(x_j))} d_{jk}(\tau,x,t) 
\beta_{jk}\left(x_{jk},\om_j(x_{jk})\right)
\frac{(a_ka_j)(x_{jk},\om_j(x_{jk}))}{\d_3\om_k(x_k,x_{jk},\om_j(x_{jk}))}
\left(PDu^l\right)_k(\tau)
\,d\tau
\nonumber\\
&&\displaystyle
=\sum_{k\not=j}^n\int_{\om_k(x_k)}^{\om_k(x_k,x_j,\om_j(x_j))} d_{jk}(\tau,x,t) 
\frac{(\beta_{jk}a_ka_j)(x_{jk},\om_j(x_{jk}))}{\d_3\om_k(x_k,x_{jk},\om_j(x_{jk}))}
\sum\limits_{s=1}^np_{ks}(Du^l)_s(1-x_s,\tau)
\,d\tau\nonumber\\
&&\displaystyle
=\sum_{k\not=j}^n\int_{\om_k(x_k)}^{\om_k(x_k,x_j,\om_j(x_j))} d_{jk}(\tau,x,t) 
\frac{(\beta_{jk}a_ka_j)(x_{jk},\om_j(x_{jk}))}{\d_3\om_k(x_k,x_{jk},\om_j(x_{jk}))}
\nonumber\\
&&\displaystyle
\times
\sum\limits_{s=1}^np_{ks}
\int_{1-x_s}^{x_s} d_s(\xi,1-x_s,\tau)\sum_{r\not=s}(b_{sr}u^l_r)(\xi,\om_s(\xi,1-x_s,\tau))\,d\xi
\,d\tau,\label{DSD}
\end{eqnarray}
where
$$
d_{jk}(\tau,x,t)=d_j(x_{jk}(\tau,x,t),x,t)
c_k\left(x_k,x_{jk}(\tau,x,t),\om_j(x_{jk}(\tau,x,t))\right)
$$
and  $x_{jk}=x_{jk}(\tau,x,t)$ denotes the inverse map to 
$\xi\mapsto\tau=\om_k(x_k,\xi,\om_j(\xi))$  for all  $\xi$ such that
$
a_k(\xi,\om_j(\xi))\ne a_j(\xi,\om_j(\xi)).
$

Write
\begin{eqnarray*}
\displaystyle
d_{jksr}(\xi,\tau,x,t) =d_{jk}(\tau,x,t)
\frac{(\beta_{jk}a_ka_j)(x_{jk},\om_j(x_{jk}))}{\d_3\om_k(x_k,x_{jk},\om_j(x_{jk}))}
p_{ks} d_s(\xi,1-x_s,\tau)b_{sr}(\xi,\om_s(\xi,1-x_s,\tau)).
\end{eqnarray*}
Further we proceed with an arbitrary fixed summand in the right-hand side of
\reff{DSD}. For definiteness, fix arbitrary $k\ne j$, 
$s$ in the range $m+1\le s\le n$, and  $r\ne s$. 
After  applying  Fubini's theorem to the corresponding summand in \reff{DSD}, it reads
\begin{eqnarray}
\lefteqn{
\displaystyle
\int_{0}^{1}\int_{\om_k(x_k)}^{\om_k(x_k,x_j,\om_j(x_j))}
d_{jksr}(\xi,\tau,x,t)
u^l_r(\xi,\om_s(\xi,0,\tau))\,d\tau\,d\xi}\label{one}\\
&&
=\int_{0}^{1}\int_{\om_s(\xi,0,\om_k(x_k))}^{\om_s(\xi,0,\om_k(x_k,x_j,\om_j(x_j)))}
d_{jksr}(\xi,\om_s(0,\xi,\theta),x,t)\d_3\om_s(0,\xi,\theta)
u^l_r(\xi,\theta)\,d\theta\,d\xi.\nonumber
\end{eqnarray}
Here we used the change of variables $\tau\mapsto\theta=\om_s(\xi,0,\tau)$, that is, $\tau=\om_s(0,\xi,\theta)$.
Since
$$
\left|\om_s(\xi,0,\om_k(x_k))-\om_s(\xi,0,\om_k(x_k,x_j,\om_j(x_j)))\right|\le 2d \quad \mbox{for all}\,\,(x,t)\in\overline\Pi,
$$
 the $C\left(\overline\Pi_{\tau+2d}^{\tau+2d+\al}\right)^n$-norm
of the right-hand side of \reff{one}  can be estimated from above by
  \begin{eqnarray}
&\displaystyle
2d\max\limits_{x,\xi\in[0,1]}\max\limits_{t,\theta\in\R}
\left|d_{jksr}(\xi,\om_s(0,\xi,\theta),x,t)\d_3\om_s(0,\xi,\theta)\right|
\max\limits_{\theta\in [\tau,\tau+2d+\al]}\int_0^1|u_r^l(\xi,\theta)|\,d\xi \nonumber\\
&\displaystyle
\le K\left\|u_r^l\right\|_{C([\tau,\tau+2d+\al],L^2(0,1))},\label{two}&
\end{eqnarray}
where $K$ is a constant independent of $u_r^l$. This  implies the desired convergence for 
each summand in \reff{DSD}
and, therefore,
for the whole $\left[DCPDu^l\right](x,t)$. This completes the proof of Claim 2.

The proof of \reff{conv_all}  for $\left[D(CP)^iDu^l\right](x,t)$ with $i=2,\dots,k-1$
follows the same line, since the operator $CP$ is bounded.
 It follows that any  $L^2$-generalized solution
 $u$  to the problem under consideration
is a continuous function for all $t\geq \tau+T_0$, where $T_0=2d$. 
Furthermore, the estimates
 (\ref{conv1}) and \reff{two} imply that
$$
\|u\|_{C\left(\overline\Pi_{\tau+T_0}^{\tau+T_0+\al}\right)^n}\le 
K\left\|u\right\|_{C([\tau,\tau+T_0+\al],L^2(0,1))^n},
$$
where $K$ is a constant depending on  $\alpha$, $a$ and $b$ but 
 not on $\tau$. Using additionally the estimate \reff{eq:apr2}, we come to the inequality
$$
\sup_{\tau+T_0 \le t \le \tau+T_0+\al}\|U(t,\tau)\varphi\|_{C([0,1])^n}\le K\|\varphi\|_{L^2(0,1)^n} , \quad t,\tau\in \R,
$$
where $K$ is a constant independent of $\varphi$. 
Note that, given $c>0$, the constant $K$ can be chosen the same for all 
$b_{jk}$ such that $\max_{j,k}\bigl\|b_{jk}\bigr\|_1<c$.
The proof of the lemma is complete.
\end{proof}

 The following smoothing result is proved in  \cite[Theorem 2.7]{Km}.
\begin{lem}\cite{Km}\label{lem:d1}
Suppose that the conditions  (\ref{ss44}), (\ref{ss4}), (\ref{ss8}), \reff{U-smooth}, (\ref{cass})    are fulfilled. Then the evolution family 
$\{U(t,\tau)\}_{t\ge \tau}$  generated by the problem (\ref{ss10}), (\ref{ss5})
 is
  smoothing from $Y_0$ to 
$C^1\left([0,1]\right)^n$.
\end{lem}

Theorem \ref{main_smoothing}
follows from Lemmas  \ref{lem:d} and  \ref{lem:d1}. 
The  smoothing time  $T$, after which the $L^2$-generalized solution to the problem
(\ref{ss10}), (\ref{ss5}), \reff{ss11}  becomes $C^1$-smooth is equal to
 $T_0+T_1$, where $T_0$ is the smoothing time from $L^2$- to $C$-regularity 
ensured by Lemma \ref{lem:d}
and $T_1$ is the smoothing time from $C$- to $C^1$-regularity 
ensured by \cite[Theorem 2.7]{Km}.
Furthermore, for given $\varphi\in L^2(0,1)^n$ and $\al>0$,  the following bound is fulfilled:
\begin{equation}\label{eq:nnn2}
\sup_{\tau+T \le t \le \tau+T+\al}\|U(t,\tau)\varphi\|_{C^1([0,1])^n}\le K\|\varphi\|_{L^2(0,1)^n} , 
\quad t,\tau\in \R,
\end{equation}
where $K$ is a constant that depends on  $\alpha$, $a$, $b$ and $p_{jk}$ $(j,k \le n)$ but 
 not on $\tau$ and $\vphi$.  Moreover, given $c>0$, 
the constant $K$ can be chosen the same for all 
$b_{jk}$ such that $\max_{j,k}\bigl\|b_{jk}\bigr\|_1<c$.

\section{Proof of the perturbation theorem}
\renewcommand{\theequation}{{\thesection}.\arabic{equation}}
\setcounter{equation}{0}
\subsection{Abstract setting}

Let us write down  the unperturbed and the perturbed  problems 
(\ref{ss2}), (\ref{ss11}), (\ref{ss5}) and (\ref{ss7}), (\ref{ss11}), (\ref{ss5}), 
respectively, in the form of abstract evolution equations in the Hilbert space  $L^2(0,1)^n$.
To this end, denote $v(t)=(u_{1}(0,t),\dots u_m(0,t)$, $u_{m+1}(1,t),\dots u_n(1,t))$ and   one-parameter families of operators
$A(t)$ and $B(t)$ from $L^2(0,1)^n$
to $L^2(0,1)^n$ for each $t\in \R$,
defined by
\beq\label{A}
\begin{array}{ll}
\displaystyle \left(A(t)u\right)(x)=\left(-a(x,t)\frac{\d}{\d x} - b_d(x,t)\right)u,\\ [3mm]
\left(B(t)u\right)(x)=\left( - \tilde b(x,t)\right)u,
\end{array}
\ee
where the domains of definition are given by 
$$
\begin{array}{ll}
\displaystyle
 D(A(t))=\{u\in L^2(0,1)^n\,:\,\d_xu\in L^2(0,1)^n,\,v(t)=
(Pu)(t)\},\\ [2mm]D(B(t))= L^2(0,1)^n,
\end{array}
$$
for the operator   $P$  given by  (\ref{ss12}). Note that $D(A(t)+B(t))=D(A(t))$.

In this notation, the {\it unperturbed problem} (\ref{ss2}), (\ref{ss11}), (\ref{ss5}) reads
\beq\label{unperturb}
\frac{d}{dt}u=A(t)u , \quad u(\tau)=\varphi\in L^2(0,1)^n,
\ee
while the {\it perturbed problem}  (\ref{ss7}), (\ref{ss11}), (\ref{ss5})  reads
\beq\label{perturb}
\frac{d}{dt}u=(A(t)+B(t))u, \quad u(\tau)=\varphi\in L^2(0,1)^n.
\ee
Accordingly to the above notation,   $\{U(t,\tau)\}_{t\ge\tau}$ and $\{\tilde{U}(t,\tau)\}_{t\ge\tau}$ will  denote evolution 
families on $L^2(0,1)^n$ generated by the problems (\ref{unperturb}) and 
(\ref{perturb}), respectively.

\subsection{Proof of Theorem \ref{hyper}}\label{sec:perturb}

\paragraph{Part $(\io)$}
 Consider  the 
abstract formulations  \reff{unperturb} and  \reff{perturb} of the unperturbed and 
 perturbed problems 
(\ref{ss2}), (\ref{ss11}), (\ref{ss5}) and (\ref{ss7}), (\ref{ss11}), (\ref{ss5}), 
respectively. 
By Theorem \ref{evol}, the problems
 \reff{unperturb} and \reff{perturb}   generate exponentially bounded evolution families 
 $\{U(t,\tau)\}_{t\ge\tau}$ and $\{\tilde{U}(t,\tau)\}_{t\ge\tau}$, respectively. 
The condition \reff{U-smooth} implies that
\beq\label{stab}
U(t,\tau)\varphi=0 \quad \mbox{ for all} \,\, t\geq\tau+d\quad \mbox{and}\,\, \varphi\in L^2(0,1)^n.
\ee
 Set  
$$
\be=\max_{j,k}\bigl\|\tilde b_{jk}\bigr\|_\infty.
$$
By Lemma \reff{nn0}, we have  bounds 
\begin{equation}\label{ll1}
\sup_{0\le t-\tau \le d}\|U(t,\tau)\|_{\LL(L^2(0,1)^n)}\le C_{U},\quad
\sup_{0\le t-\tau \le d}\|\tilde{U}(t,\tau)\|_{\LL(L^2(0,1)^n)}\le C_{\tilde{U}}
\end{equation}
for some positive constants $C_{U}$ and $C_{\tilde{U}}$ not depending on $\tau$.

Fix $\tau\in\R$ and $\varphi\in C_0^\infty([0,1])^n$. Then
 $U(t,\tau)\varphi$ and $\tilde U(t,\tau)\varphi$ are classical
 solutions to the problems
\reff{unperturb} and \reff{perturb}, respectively. 
This allows us to apply the variation of constants formula (see, e.g. \cite{Paz}),
which gives us the equation
\beq\label{Pazy}
\tilde{ U}(t,\tau)\varphi=U(t,\tau)\varphi+\int_{\tau}^{t}U(t,s)B(s)
\tilde {U}(s,\tau)\varphi \, ds\quad  \mbox{ for }\,\, t\geq\tau,
\ee
where $B(t)$ is determined by \reff{A}. Our aim is to  prove  the bound (\ref{bound1}).
For $t>\tau+d$, the formula (\ref{Pazy})  reads
$$
\tilde U(t,\tau)\varphi=U(t,\tau)\varphi+\int_{\tau}^{t-d}U(t,s)B(s)\tilde U(s,\tau) \varphi\, ds +
\int_{t-d}^{t}U(t,s)B(s)\tilde U(s,\tau) \varphi\, ds.
$$
By  (\ref{stab}), the first two summands in the right-hand side vanish, and we get
\begin{equation}\label{eq:nn16}
\tilde U(t,\tau)\varphi=\int_{t-d}^{t}U(t,s)B(s)\tilde U(s,\tau)\varphi\, ds\quad  \mbox{ for }\,\,  t>\tau+d.
\end{equation}
Write $Z(t)=\|\tilde U(t,\tau)\varphi\|_{L^2(0,1)^n}$. Due to (\ref{ll1}) and (\ref{eq:nn16}), 
\begin{equation}\label{ll2}
Z(t)\le C_{\tilde{U}}\|\varphi\|_{L^2(0,1)^n}\quad  \mbox{ for }\,\,  \tau\le t\le \tau+d,
\end{equation}
\begin{equation}\label{ll3}
Z(t)\le C_{U}\be \int _{t-d}^{t} Z(s) ds\quad  \mbox{ for }\,\, \tau+d< t.
\end{equation}

\begin{lem}\label{lem:lll1}
Suppose that the function $Z(t)$ fulfills the estimates (\ref{ll2}), (\ref{ll3}). 
Then for any  $\ga>0$ there exists  $\eps>0$  
such that 
\beq\label{lll4}
Z(t)\le 
M e^{-\ga(t-\tau)}\|\varphi\|_{L^2(0,1)^n}\quad  \mbox{ for }\,\,  t\ge\tau,
\ee
for all  $\be\in (0,\varepsilon)$, where $M=C_{\tilde{U}}e^{(\ga+C_{U}\eps)d}$.
\end{lem}
\begin{proof}
Fix an arbitrary $\ga>0$. If  $\tau < t\le \tau+d$, then the desired estimate 
\reff{lll4} follows from \reff{ll2}. Indeed, 
\begin{eqnarray*}
&Z(t)\le  C_{\tilde U}\|\varphi\|_{L^2(0,1)^n}=C_{\tilde U}
e^{\ga(t-\tau)}e^{-\ga(t-\tau)}\|\varphi\|_{L^2(0,1)^n}\le Me^{-\ga(t-\tau)}\|\varphi\|_{L^2(0,1)^n}&\\ &  \hskip10cm \,\, \mbox{ for }\,\, 0 < t-\tau\le d.&
\end{eqnarray*}

If $\tau+d < t\le \tau+2d$, then  $\tau < t-d\le \tau+d<t$ and,
 due to  (\ref{ll2}) and (\ref{ll3}),
\begin{eqnarray*}
\lefteqn{
Z(t)\le  C_{U}\be \int _{t-d}^{t} Z(s) ds =   C_{U}\be
\left[\int_{t-d}^{\tau+d} Z(s) ds+\int _{\tau+d}^{t} Z(s) ds\right]}\\  
&&
\le
 C_{U}\be
\left[C_{\tilde{U}}d\|\varphi\|_{L^2(0,1)^n}+\int _{\tau+d}^{t} Z(s) ds\right]\\  
&&
\le C_{U}C_{\tilde{U}}\be  d\|\varphi\|_{L^2(0,1)^n}+C_{U}\be\int _{\tau+d}^{t} Z(s) ds.
\end{eqnarray*}
By Gronwall's argument, we obtain
\beq\label{lll5}
Z(t)\le C_{U}C_{\tilde{U}}\be  d
e^{C_{U}\be(t-\tau-d)}\|\varphi\|_{L^2(0,1)^n}\quad  \mbox{ for }\,\,  d < t-\tau\le 2d.
\ee

If  $\tau+2d<t\le \tau+3d$, then $\tau+d<t-d\le \tau+2d<t$ and, on the account of
(\ref{ll3}) and  (\ref{lll5}), we come to the inequality 
\begin{eqnarray*}
\lefteqn{
Z(t)\le  C_{U}\be \int _{t-d}^{t} Z(s) ds =   C_{U}\be
\left[\int_{t-d}^{\tau+2d} Z(s) ds+\int _{\tau+2d}^{t} Z(s) ds\right]}\\  
&&
\le   
 C_{\tilde{U}}(C_{U}\be  d)^2 e^{C_{U}\be d} \|\varphi\|_{L^2(0,1)^n}
+C_{U}\be\int _{\tau+2d}^{t} Z(s) ds.
\end{eqnarray*}
Again, the Gronwall's argument gives 
$$
Z(t)\le  C_{\tilde{U}} (C_{U}\be d)^2
e^{C_{U}\be(t-\tau-d)}\|\varphi\|_{L^2(0,1)^n}\quad  \mbox{ for }\,\,  2d < t-\tau\le 3d.
$$

Proceeding further by induction, on the $k$-th step we obtain an estimate for the function 
 $Z(t)$, namely
$$
Z(t)\le C_{\tilde{U}}(C_{U}\be d)^k 
e^{C_{U}\be(t-\tau-d)}\|\varphi\|_{L^2(0,1)^n}\quad  \mbox{ for }\,\, 
kd < t-\tau\le (k+1)d.
$$
Fix $\eps$  to  fulfill the equality 
\beq\label{*2}
\frac{1}{d}\log(C_{U}\eps d)+C_{U}\eps=-\ga.
\ee
Note that $\log(C_{U}\eps d)<0$. Since $t-\tau-d\le kd$ and $\be<\eps$, we get
$$
\begin{array}{rcl}
Z(t)&\le& C_{\tilde U}
e^{kd\frac{1}{d}\log(C_{U}\eps d)+C_{U}\eps(t-\tau)}
 \|\varphi\|_{L^2(0,1)^n}\nonumber\\ [3mm] &\le &
C_{\tilde U}e^{(t-\tau-d)\frac{1}{d}
\log(C_{U}\eps d)+C_{U}\eps(t-\tau)}
 \|\varphi\|_{L^2(0,1)^n}\nonumber\\ [3mm] &= &
C_{\tilde U}e^{-
\log(C_{U}\eps d)}e^{-\ga(t-\tau)}
 \|\varphi\|_{L^2(0,1)^n}\nonumber\\ [3mm] &= &
 C_{\tilde U}
e^{(\ga+C_{U}\eps)d}e^{-\ga (t-\tau)} \|\varphi\|_{L^2(0,1)^n}\quad  \mbox{ for }\,\,  kd < t-\tau\le (k+1)d,
\end{array}
$$
where the last equality holds by  \reff{*2}.
Since $k\in\N$ is arbitrary,  the estimate 
(\ref{lll4})  follows. The proof of the lemma is  complete. 
\end{proof}

 Lemma \ref{lem:lll1} gives the estimate 
$$
\|\tilde U(t,\tau)\varphi\|_{L^2(0,1)^n}\le Me^{-\gamma(t-\tau)}\|\varphi \|_{L^2(0,1)^n} \quad 
\mbox{ for }\,\,t\ge \tau,
$$
for all $\varphi\in C_0^\infty([0,1])^n$. Since the space $C_0^\infty([0,1])^n$ is dense in
 $L^2(0,1)^n$, the same estimate is true for all $\varphi\in L^2(0,1)^n$. 
This entails (\ref{bound1}), as desired.
The proof of part $(\io)$ is complete.

\paragraph{Part $(\io\io)$}
Fix an arbitrary $\ga>0$. Let $\eps$ be a positive real satisfying the estimate  \reff{bound1} for all 
 $\tilde b_{jk}$
such that $\max_{j,k}\bigl\|\tilde b_{jk}\bigr\|_\infty<\eps$.

Recall that all  assumptions of Theorem~\ref{main_smoothing} are stable with respect to certain perturbations of $b_{jk}$. 
Specifically, the condition \reff{U-smooth} is stable, because it  depends only on the diagonal part of $b$, while
the perturbations involve only the non-diagonal part of $b$. The condition \reff{cass} is true by the 
corresponding assumption 
 of Theorem  \ref{hyper} $(\io\io)$. 

 Therefore, by  Theorem~\ref{main_smoothing}, 
$\tilde U(t,\tau)$ is smoothing  from  $L^2(0,1)^n $ to $C^1([0,1])^n$.
More specifically,
 for any 
 $\varphi \in L^2(0,1)^n$ the  function $\tilde U(t,\tau)\varphi$
 belongs to $C^1([0,1])^n$  whenever $t \ge \tau+T$ for some $T>0$. 
Furthermore,  the estimate   (\ref{eq:nnn2})
implies that
\begin{equation}\label{eq:nn11}
\|\tilde U(\tau+T,\tau)\varphi\|_{C^1([0,1])^n}\le K_0\|\varphi\|_{L^2(0,1)^n}  \quad  \mbox{ for all } 
\tau\in \R
\end{equation}
for some $K_0\ge 1$ not depending on $\varphi$ and $\tau$.
 Moreover, the constant $K_0$ can be chosen the same for all $\tilde b_{jk}$
such that $\max_{j,k}\bigl\|\tilde b_{jk}\bigr\|_1<\eps$.

Suppose that  
 $\varphi \in C^1([0,1])^n$ satisfies the zero-order and the first-order 
compatibility conditions 
 (\ref{eq:nl1}) and (\ref{eq:nl2}) for the perturbed problem
 \reff{perturb}. By  
Theorem \ref{km},
we have
\begin{equation}\label{eq:nn12}
\|\tilde U(t,\tau)\varphi\|_{C^1([0,1])^n}\le K_T \|\varphi\|_{C^1([0,1])^n} \quad \mbox{ for all } \tau \le t \le \tau+2T,
\end{equation}
where $K_T=K_1e^{2T\omega_1}>1$ for the constants $K_1$ and $\om_1$  as in
 (\ref{eq:apr5}). 
Again, the constant $K_1$ and, hence $K_T$ can be chosen the same for all $\tilde b_{jk}$
such that $\max_{j,k}\bigl\|\tilde b_{jk}\bigr\|_1<\eps$.

Fix an arbitrary $\tau \in \R$. To prove the estimate (\ref{bound2}), 
it suffices to show that 
$$
\|\tilde U(t,\tau)\varphi \|_{C^1([0,1])^n}\le 
M_1 e^{-\ga(t-\tau)}\|\varphi\|_{L^2(0,1)^n}\quad  \mbox{ for all } t\ge\tau+2T
$$
for all $\varphi \in L^2(0,1)^n$ and  some  $M_1\ge M$, where $M$ fulfills \reff{bound1}.

If $t \ge \tau +2T$, then there is  $k\ge 2$ such that 
$
\tau+kT \le t <\tau+(k+1)T.
$
Then
$
t-\tau-(k-1)T\le 2T.
$
Taking  (\ref{eq:nn12}) into account, we see that
\begin{equation}\label{eq:nn22}
\begin{array}{rcl}
\|\tilde U(t,\tau)\varphi\|_{C^1([0,1])^n}&=&
\|\tilde U(t,\tau+(k-1)T)\tilde U(\tau+(k-1)T,\tau)\varphi\|_{C^1([0,1])^n}\\
& \le &
K_T \|\tilde U(\tau+(k-1)T,\tau)\varphi\|_{C^1([0,1])^n}.
\end{array}
\end{equation}
The  estimates  (\ref{eq:nn11}) and (\ref{bound1}) imply that for $k\ge 2$
\beq\label{bbb4}\begin{array}{ll}
 \|\tilde U(\tau+(k-1)T,\tau)\varphi\|_{C^1([0,1])^n}\\\quad
= 
\|\tilde U(\tau+(k-1)T,\tau+(k-2)T )\tilde U(\tau+(k-2)T,\tau)\varphi\|_{C^1([0,1])^n}\\\quad
\le K_0 \|\tilde U(\tau+(k-2)T,\tau)\varphi\|_{L^2(0,1)^n}\\\quad
\le K_0M e^{-\gamma (k-2)T}\|\varphi\|_{L^2(0,1)^n}.
\end{array}
 \ee 
Finally, combining the estimates (\ref{eq:nn22}) and  (\ref{bbb4}),  we get
$$\begin{array}{rcl}
\|\tilde U(t,\tau)\varphi\|_{C^1([0,1])^n}
\le K_T K_0 M e^{-\gamma (k-2)T}\|\varphi\|_{L^2(0,1)^n}&=& K_T K_0 M e^{3\gamma T}e^{-\gamma (k+1)T}\|\varphi\|_{L^2(0,1)^n}\\
& \le& M_1 e^{- \gamma (t-\tau)}\|\varphi\|_{L^2(0,1)^n},
\end{array}
$$
where  $M_1=K_T K_0 M e^{3\gamma T}>M$ and $t \ge \tau +2T$.  

The proof of part  $(\io\io)$ of Theorem  \ref{hyper} is complete.

(


\begin{thebibliography}{10}

\bibitem{Aboliny}
V.E. Abolinya, A.D. Myshkis,  A mixed problem for an almost linear hyperbolic system on the plane,
{\it Matematicheskij Sbornik} {\bf 50(92)}(4) (1960), 423--442. 

\bibitem{bastin}G. Bastin,  J.-M. Coron, {\it Stability and Boundary Stabilization of 1-D Hyperbolic Systems},
Progress in Nonlinear Differential Equations and Their Applications {\bf 88}, 
Birkh\"auser, 2016.


\bibitem{bal99} A.V. Balakrishnan, On superstability of semigroups. Systems modelling and optimization, in: M.P.Polis et al. (Eds.), Proceedings of the 18th
IFIP TC7 Conference on System Modelling and Optimization, CRC, Research Notes in Mathematics, Chapman and Hall (1999),  12--19.

\bibitem{bal105} A.V. Balakrishnan, Superstability of systems,
{\it  Applied Mathematics and Computation} {\bf 164(2)} (2005), 321--326.


\bibitem{chen} J.-H. Chen, W.-Y. Lu,   Perturbation of nilpotent semigroups and application to heat exchanger equations, {\it Applied Mathematics Letters} {\bf 24} (2011), 1698--1701.

\bibitem{Coron} J.-M. Coron, G. Bastin, Dissipative boundary conditions 
for one-dimensional quasilinear hyperbolic systems: Lyapunov stability
for the $C1$-norm, {\it SIAM J. Control Optim.} {\bf 53(3)} (2015), 1464--1483.

\bibitem{Cox} S. Cox, E. L. Zuazua,  The rate at which energy decays in a string damped at one end,
{\it  Indiana Univ. Math. J.} {\bf 44(2)} (1995), 545--573.

\bibitem{cr13}  D. Creutz, M. Mazo Jr., C. Preda, Superstability and finite time  extinction  
for $C_0$-semigroups, (2013). {\it  E-print}: https://arxiv.org/abs/0907.4812.

\bibitem{els} L.E. Elsgolts, S.B. Norkin.
An Introduction to the Theory and Application of Differential
Equations with Deviating Arguments. Academic Press, New York, 1973.


\bibitem{Elt} N.A.~\"Eltysheva,  On qualitative properties of solutions to
some hyperbolic systems on the plane,
{\it Matematicheskij Sbornik} {\bf 135(2)} (1988), 186--209.



 \bibitem{god} S.K. Godunov, {\it Equations of Mathematical Physics}, Moscow: Nauka,   2nd ed., 1979 
(Russian).

\bibitem{gugat} M. Gugat, Boundary feedback stabilization of the telegraph equation: Decay rates for vanishing damping term,
{\it Systems and Control Letters} {\bf 66} (2014), 72--84. 

\bibitem{gugat1} M. Gugat, {\it Optimal Boundary Control and Boundary Stabilization of Hyperbolic Systems},
Basel : Birkh\"auser, 2015.

\bibitem{hp57} E. Hille, R. Phillips,  {\it Functional  analysis and semi-groups},
Providence, 1957.

\bibitem{ijdsde}
I. Kmit,  Classical solvability of nonlinear initial-boundary problems for
first-order hyperbolic systems, {\it Intern. J. Dynamic Systems Different.
Equat.} {\bf 1(3)} (2008), 191--195.

\bibitem{kmit} I. Kmit,  Smoothing effect and Fredholm property for 
first order hyperbolic PDEs,   
In: {\it Operator Theory: Advances and Applications}, Basel: Birkh\"auser
 {\bf 231} (2013), 219--238.

\bibitem{Km} I. Kmit, Smoothing solutions to initial-boundary problems  for first-order
hyperbolic systems, {\it Applicable Analysis} {\bf 90(11)} (2011), 1609--1634.

\bibitem{Km1} I. Kmit, G. H\"ormann, Systems with singular non-local boundary conditions: Reflection of singularities and delta waves, {\it J. Anal. Appl.} {\bf 20(3)} (2001), 637--659.

\bibitem{Kom} V. Komornik, Rapid boundary stabilization of the wave equation,
{\it  SIAM J. Control Optim.} {\bf 29} (1991), 197--208.


\bibitem{kr} M.G. Krein, I.C. Gohberg, {\it Theory and applications of Volterra operators in Hilbert space}, 
American Math Society, 1970.

\bibitem{LavLyu} M.M.~Lavrent'ev Jr., N.A.~Lyul'ko,  Increasing
smoothness of solutions to some hyperbolic problems,
{\it Siberian Math. J.}
 {\bf 38(1)} (1997), 92--105.


\bibitem{Lyu} N.A. Lyul'ko,
Increasing smoothness of solutions to a hyperbolic system on the plane with delay in the boundary 
conditions. 
{\it Siberian Math. J.} {\bf 49(6)} (2008), 1333--1350.

\bibitem{lum01} G. Lumer, On the growth of the resolvents for an explicit class  of superstable  semigroups.
 Ulmer Seminare  uber Funktionalanalysis  and  Differentialgleichungen,
{\it Appl. Analysis}, Univ. Ulm {\bf 6} (2002), 253--258.

\bibitem{Majda}  A. Majda,
Disappearing solutions for the dissipative wave equation,
{\it Indiana Univ. Math. J.} {\bf 24} (1975), 1119--1133.


\bibitem{lakra} 
L. Pavel,
Classical solutions in Sobolev spaces for a class of hyperbolic Lotka--Volterra systems,
{\it SIAM J. Control Optim.} {\bf 51(3)} (2013), 2132--2151. 


\bibitem{pavel} L. Pavel, L. Chang,
Lyapunov-based boundary control for a class of hyperbolic Lotka-
Volterra systems, {\it IEEE Trans. Automat. Control} {\bf 7} (2012), 701--714.

\bibitem{Paz} 
A. Pazy,
{\it Semigroups of operators and applications to partial differential equations},
Springer-Verlag, Berlin,  1983. 

\bibitem{perroll} V. Perrollaz, L. Rosier, Finite-Time Stabilization of $2\times2$ Hyperbolic Systems on Tree-Shaped Networks, {\it  SIAM Journal on Control and Optimization} {\bf 52(1)}
(2014), 143--163.


\bibitem{rw95} F. R\"abiger and M. Wolff,   Superstable semigroups of operators,
{\it  Indagationes Mathematicae} {\bf 6} (1995), 481--494.

\bibitem{Rom} 
R.K. Romanovskiy, E.V. Vorobyova, E.N. Stratilatova,
{\it The Riemann method for hyperbolic systems},
Novosibirsk: Nauka,  2007 (Russian).



\bibitem{slx13} Y. Shang, D. Liu, G. Xu, Super-stability and the spectrum of one-dimensional wave equations  on general feedback controlled networks, {\it  IMA Journal of  Mathematical Control and Information} {\bf 31(1)} (2014), 73--99. 


\bibitem{udw05} F.E. Udwadia, Boundary control, quiet boundaries, super-stability and super-instability, 
{\it  Applied Mathematics and Computation} {\bf 164(2)} (2005), 327--349.

\bibitem{udw12} F.E. Udwadia, On the longitudinal  vibrations of a bar with  viscous  boundaries: Super-stability, super-instability 
and loss damping,
{\it  Intern. J. of Engineering Science} {\bf 50(1)} (2012), 79-100.


\bibitem{Zel} T.I. Zelenyak, 
On stationary solutions of mixed problems relating to the study of certain chemical processes, 
{\it Differ. Equations} {\bf 2(2)} (1966), 205--213. 

\bibitem{Zel1} T.I. Zelenyak, On the question of stability of mixed problems for a quasi-linear equation
 (Russian), 
{\it Differ. Equations} {\bf 3(1)} (1967), 19--29. 




\end{thebibliography}
\end{document}